\documentclass[12pt, reqno]{amsart}
\usepackage{amssymb}
\usepackage[a4paper, margin=.9in, includeheadfoot]{geometry}
\usepackage{geometry}
\usepackage[english]{babel}
\usepackage{palatino}
\usepackage{graphicx}
\usepackage[ocgcolorlinks]{hyperref}
\hypersetup{
	pdfstartview=FitH,
	bookmarksopen=false,
	pdfencoding=auto,
	colorlinks=false,
	linkcolor=blue,
	pdfinfo={
		Title={The Flow of Gauge Transformation},
		Subject={58J35, 58E15},
		Author={Van Abel},
		Keywords={Heat flow, Coulomb gauge, blow-up analysis}%
	}
}
\usepackage[backrefs, msc-links, lite, abbrev]{amsrefs} %author-year, alphabetic

\newtheorem{mthm}{Theorem}

\newtheorem{mcor}[mthm]{Corollary}
\newtheorem{mlem}[mthm]{Lemma}

\newtheorem{defn}{Definition}[section]
\newtheorem{thm}[defn]{Theorem}

\newtheorem{lem}[defn]{Lemma}
\newtheorem{cor}[defn]{Corollary}
\newtheorem{prop}[defn]{Proposition}
\newtheorem*{rmk}{Remark}
\numberwithin{equation}{section}

\newcommand{\Aut}{\operatorname{Aut}\!}
\newcommand{\ga}{\alpha}
\newcommand{\gb}{\beta}
\newcommand{\gc}{\gamma}
\newcommand{\CA}{\mathcal{A}}
\newcommand{\CM}{\mathcal{M}}
\newcommand{\B}{\mathfrak{B}}
\newcommand{\abs}[1]{\left|#1\right|}

\newcommand{\bpar}[1]{\left(#1\right)}
\newcommand{\comp}{\raisebox{1pt}{\scalebox{.7}{$\circ$}}}

\newcommand{\dist}{\mathop{\operatorname{dist}}}
\newcommand{\dt}[1]{\frac{\rd{#1}}{\rd t}}
\newcommand{\ddt}[1]{\left.\frac{\pt}{\pt t}\right|_{t={#1}}}

\newcommand{\E}{\mathfrak{E}}
\newcommand{\embedto}{\hookrightarrow}
\newcommand{\End}{\operatorname{End}}
\newcommand{\eng}{\mathcal{E}}
\newcommand{\eps}{\varepsilon}
\newcommand{\epst}{\texorpdfstring{$\eps$}{epsilon}}
\newcommand{\eqdef}{\mathpunct{:}=}

\newcommand{\g}{\mathfrak{g}}

\newcommand{\heatop}{\left(\ppt{t}-\Delta\right)}

\newcommand{\id}{\operatorname{Id}}
\newcommand{\II}{\mathrm{I\!I}}
\newcommand{\inner}[1]{\left\langle#1\right\rangle}
\newcommand{\lh}{\,\rule{2.5mm}{0.5pt}\rule{0.5pt}{1.5mm}\,}
\newcommand{\norm}[1]{\left\|{#1}\right\|}
\newcommand{\osc}{\operatorname{osc}}
\newcommand{\R}{\mathbb{R}}

\newcommand{\rd}{\operatorname{d}}
\renewcommand{\S}{\mathcal{S}}
\newcommand{\set}[1]{\left\{#1\right\}}
\newcommand{\SO}{\mathrm{SO}}
\newcommand{\so}{\mathfrak{so}}
\newcommand{\tr}{\operatorname{tr}}
\newcommand{\pt}{\partial}
\newcommand{\ppt}[1]{\frac{\pt}{\pt{#1}}}

\newcommand{\vppt}[2]{\frac{\pt{#1}}{\pt{#2}}}
\newcommand{\veps}{\texorpdfstring{$\eps$}{Epsilon}}
\newcommand{\weakto}{\rightharpoonup}
\newcommand{\xkf}[1]{\left(#1\right)}
\newcommand{\zkf}[1]{\left[#1\right]}
\newcommand{\dkf}[1]{\left\{#1\right\}}

\def\XXint#1#2#3{{\setbox0=\hbox{$#1{#2#3}{\int}$ }
\vcenter{\hbox{$#2#3$ }}\kern-.6\wd0}}

\newenvironment{abst}{%
	\begin{abstract}%
	}{%
	\end{abstract}%
}
\newenvironment{eq}{\equation}{\endequation}

\title[Flow of Gauge Transformations]{The Flow of Gauge Transformations on Riemannian Surface with Boundary}
\author{Wanjun Ai}
\address{School of Mathematical Sciences, University of Science and Technology of China}%
\email{aiwanjun@mail.ustc.edu.cn}%

\thanks{}%
\date{\today}

\begin{document}
%\allowdisplaybreaks[4]
\begin{abst}
	We consider the gauge transformations of a metric $G$-bundle over a compact Riemannian surface with boundary. By employing the heat flow method, the local existence and the long time existence of generalized solution are proved. 
\end{abst}
\maketitle
\section{Introduction and Statement of Main Results}
Let $E$ be a metric vector bundle of rank $r$ with compact structure group $G$ over a Riemannian manifold $\Sigma$. Given two fixed $G$-connections $A_0$ and $A$ on $E$, for any gauge transformation $S$ of the bundle, we define the energy
\[
	\eng(S)\eqdef\int_\Sigma e(S)=\frac{1}{2}\int_\Sigma|S^*(A)-A_0|^2,
\]
which measures the $L^2$ distance between one connection $A_0$ and the other one modified by a gauge transformation $S$. Hence, if $S$ is a critical point of $\eng(S)$, then $S^*(A)$ is in a \emph{natural} position in its orbit of gauge transformations relative to $A_0$. In fact, the \emph{Euler-Lagrange} equation of $\eng(S)$ (forgetting the boundary value problem for the moment) can be written as
\begin{eq}\label{eq:meq}
	\nabla^*_{S^*(A)}(S^*(A)-A_0)=0.
\end{eq}

In terms of geometry, \eqref{eq:meq} means that $S^*(A)-A_0$ is perpendicular to the orbit of gauge group action at $S^*(A)$, because any tangent vector of the orbit at $S^*(A)$ is given by $\nabla_{S^*(A)} \xi$ for a section $\xi$ of the bundle $\g_E$. Either by some geometric argument, or by direct computation (see \eqref{eq:conj_star_diff}), we know that $S^*(A)-A_0$ is also perpendicular to the orbit of gauge group action at $A_0$, i.e.
\begin{equation}\label{eq:gauge_fixing}
	\nabla^*_{A_0}(S^*(A)-A_0)=0.
\end{equation}

When $\Sigma$ is the unit ball $B$ of Euclidean space, $E$ is the trivial bundle $B\times \R^r$ and $A_0$ and $A$ are given by $d$ and $d+\Omega$ for some matrix-valued one form $\Omega$, \eqref{eq:gauge_fixing} reduces to the well known Coulomb gauge condition
\begin{equation*}
	d^*(S^{-1}dS + S^{-1}\Omega S)=0.
\end{equation*}

Although Uhlenbeck proved the local existence of Coulomb gauge by using the method of continuity  in \cite{Uhlenbeck1982Connections}*{Thm.~2.1}, the variational formulation is not new. In \cite{Helein2002Harmonic}*{Lem.~4.1.3}, H\'elein studied a version of $\eng(S)$, obtained a Coulomb gauge by the method of direct minimizing and applied it to the study of the regularity issue of harmonic maps. Both the method of continuity and the direct minimizing method have been generalized to various other situations and the Coulomb gauges thus found were essential in many applications, for example, to the partial regularity of stationary bi-harmonic maps (extrinsic or intrinsic) \citelist{\cite{Wang2004Biharmonic}*{Prop.~3.2} \cite{Wang2004Stationary}*{Prop.~3.1}}; To the partial regularity of Yang-Mills fields satisfying some approximability properties \cite{MeyerRiviere2003partial}*{Thm.~1.3} and to the removable of singularities of admissible Yang-Mills fields \cite{TaoTian2004singularity}*{Thm.~4.6}; To the partial regularity of elliptic systems with skew-symmetric structure  \citelist{\cite{Riviere2007Conservation}*{Lem.~A.4} \cite{RiviereStruwe2008Partial}*{Lem.~4.1}}; To the regularity theory of  Rivi\`ere-Uhlenbeck decomposition \cite{MullerSchikorra2009Boundary}*{Lem.~2.1}. We refer to \citelist{\cite{FrohlichMuller2011existence}*{Prop.~2}\cite{Schikorra2010remark}*{Thm.~2.1}} for direct minimizing method.

The purpose of this paper is to study the negative gradient flow of $\eng$ when $\Sigma$ is a two dimensional Riemannian surface with (possibly empty) boundary. The existence of a Coulomb gauge follows naturally when we study the long time behavior of the flow. More precisely, we study the following evolution problem
\begin{eq}
	\label{eq:mflw}
	\begin{cases}
		S^{-1}\vppt{S}{t}=-\nabla^*_{S^*(A)}(S^*(A)-A_0),
		& (x,t)\in \Sigma\times(0,T), \\
		S(x,0)=S_0(x), & x\in \Sigma, \\
		\nu\lh(S^*(A)-A_0)=0, & (x,t)\in\pt \Sigma\times[0,T). \\ %]
	\end{cases}
\end{eq}
Here $A$ and $A_0$ are as before, $S_0$ is the initial gauge transformation, $\nu$ is the normal vector of $\partial \Sigma$ and $\lh$ gives the pairing of a vector and a one form. We will assume the zero order compatibility condition holds (see, e.g., \cite{LadyzenskajaSolonnikovUralceva1988Linear}*{Sec.~I\!V, p.~318ff}), i.e., 
\begin{equation}\label{eq:compatibility}
	\nu\lh(S_0^*(A)-A_0)=0.
\end{equation}

Theoretically, other boundary conditions, such as Dirichlet boundary condition, are possible. The primary reason that we have chosen this one is because in the local setting we have just considered ($E=B\times \R^r$, $A_0=d$ and $A=d+\Omega$), it reduces to
\begin{equation}\label{eq:boundary}
	\nu \lh (S^{-1} dS + S^{-1}\Omega S)=0  \quad \mbox{on} \quad \partial B,	
\end{equation}
which is exactly the boundary condition appeared in \cite{Uhlenbeck1982Connections}*{Thm.~2.1, Cond. (b)}. It is also clear from \eqref{eq:boundary} that \eqref{eq:mflw} is an oblique boundary value problem.

The discussion of \eqref{eq:mflw} is similar to the harmonic map flow from surface, which was first studied by Struwe \cite{Struwe1985evolution} in the case of closed surface, by Chang K.-C. \cite{Chang1989Heat} in the case of Dirichlet boundary condition and by Ma Li \cite{Ma1991Harmonic} in the case of free boundary value condition. This is because if we put everything in local coordinates and frames, $S$ becomes a map into Lie group $G$, the highest order term in $\eng(S)$ is the harmonic map energy of this map and the main term in \eqref{eq:mflw} is the same as the harmonic map flow of a map into $G$ (see \eqref{eq:pullback}). Therefore, the proofs presented in this paper follow ideas well known in the study of harmonic map flow, while being made complicated by the lack of global frame and the oblique boundary value condition.

\begin{mthm}\label{mthm:A}
	Let $E$, $G$, $A$, $A_0$, $\Sigma$ be as above and $S_0\in C^\infty(\Gamma(\Aut_GE))$ (smooth section of the fiber bundle of gauge transformations) with $\eng(S_0)< +\infty$. If the compatibility condition \eqref{eq:compatibility} holds, then
	\begin{enumerate}
		\item\label{lab:1_mthm:A} there exist some $T_1>0$, $\alpha\in(0,1)$ and $S\in C^\infty((0,T_1)\times\Sigma, \Aut_G E)\cap C^{2,\alpha}([0,T_1)\times\Sigma,\Aut_GE)$, which solves \eqref{eq:mflw}.%]
		\item\label{lab:2_mthm:A} Moreover, if $T_1< +\infty$, then there exist finite many (blowup) points $\set{x_i}_{i=1}^{N}$, such that $S(t)\to S(T_1)$ in $C^\infty_{loc}(\Sigma\setminus\set{x_i}_{i=1}^N,\End E)$,
			where $\End E$ is the endomorphism bundle of $E$ and $S(\cdot, T_1)$ is the $W^{1,2}$-weak limit of $S(\cdot,t)$ as $t\nearrow T_1$.
		\item\label{lab:3_mthm:A} For each blowup point $x_i$ and \emph{any} $t_k\nearrow T_1$, there exist some $x_i^k\to x_i$, $\lambda_i^k\to0$ and a bubble $\omega_i\in C^\infty(\R^2, G)$, which is a harmonic map, such that as $k\to\infty$,
			\[
				w_i^k(x)\eqdef u(x_i^k+\lambda_i^kx, t_k)\to\omega_i,\quad\text{in $C^\infty_{loc}(\R^2,G)$},
			\]
			where $u\mathpunct{:}B_1(x_i)\to G$ is the local expression of $S$ (under a fixed trivialization).
		\item\label{lab:4_mthm:A} Finally, there holds the following \emph{energy inequality}
			\[
				\liminf_{k\to\infty}\eng(S(\cdot,t_k))\geq\eng(S(\cdot,T_1))
				+\sum_{i=1}^N\frac{1}{2}\int_{\R^2}|\nabla\omega_i|^2.
			\]
	\end{enumerate}
\end{mthm}
If $T_1$ is finite, then $S(T_1)$ (being smooth away from finitely many points) has bounded $\eng$ energy, or equivalently bounded $W^{1,2}$ norm. As in the theory of the harmonic map flow (see \cite{Struwe1985evolution}*{p.~576} and \cite{Ma1991Harmonic}*{p.~295}), one can apply an approximation theorem due to \cite{SchoenUhlenbeck1982regularity}*{Lem.~3.2} to find a sequence of smooth initial data $\set{S_k}$, from which a sequence of flow solutions is known and converges to a flow solution starting from $S(T_1)$.  Here, we proceed in a slightly different way. In an attempt to understand the singularity of $u(T_1)$ (the weak limit of the harmonic map flow at the first blow-up time), Qing \cite{Qing2003remark}*{Thm.~1.1} proved some refined estimates around the singular point (see also Prop.~\ref{prop:osc_estimats}). Although $u(T_1)$ is not continuous in general (by an example constructed by Topping \cite{Topping2004Winding}*{Thm.~1.14}), this estimate allows us to construct an explicit approximation of $S(T_1)$ by using a cut-off function.

\begin{mlem}[Approximation Lemma]\label{mthm:B}
	For $S$, $T_1$, $\set{x_i}_{i=1}^N$ (blowup points) as in Theorem~\ref{mthm:A} and any $\sigma>0$, there exists a smooth $\sigma$-approximation of $S(T_1)$. That is, there exist $\delta_i>0$ and a gauge transformation $\tilde S$, such that $S(T_1)=\tilde S$ in $\Sigma\setminus\bigcup_{i=1}^N B_{\delta_i}(x_i)$ and
	\begin{equation*}
		\abs{\eng(S(T_1))-\eng(\tilde S)}<  \sigma,\quad
		\nu\lh(\tilde S^*(A)-A_0)=0.
	\end{equation*}
\end{mlem}

We can restart the flow from $\tilde S$ by using Theorem~\ref{mthm:A} until we meet the next singular time and repeat the above argument to get a piecewise smooth solution for $t\in [0,\infty)$. The modification will stop after finite steps, since if one takes $\sigma$ smaller than the smallest energy of harmonic $S^2$ in $G$, the energy along the piecewise smooth flow drops by a fixed amount after each singular time. The limit $t\to \infty$ (after some possible blow-ups) of the flow will be a solution to \eqref{eq:meq} and due to the removable of singularity theorem (see Lem.~\ref{lem:rmv_sing} and Cor.~\ref{cor:rmv_bdy_sing}), it is smooth on $\Sigma$ including the boundary. We summarize this in the next theorem. %]
\begin{mthm}\label{mthm:C}
	Let $\eps_0$ be the constant in $\eps$-regularity (see Thm.~\ref{thm:epsreg}). Then for any $\sigma< \eps_0$, there exists $S(x,t)$ which is piecewise smooth on $K+1$ intervals $I_k=[T_k,T_{k+1})$, $T_0=0$, $T_{K+1}=\infty$ and solves \eqref{eq:mflw} piecewise, %]
	\[\begin{cases}
			S^{-1}\vppt{S}{t}=-\nabla^*_{S^*(A)}(S^*(A)-A_0),&(x,t)\in\Sigma\times I_k\\
			S(x,T_k)=S_k(x),&x\in\Sigma,\\
			\nu\lh(S^*(A)-A_0)=0,&(x,t)\in\pt\Sigma\times I_k,
	\end{cases}\]		
	where $S_k\in C^\infty(\Gamma(\Aut_G E))$, $S_0$ is given in Theorem \ref{mthm:A}, and for $1\leq k< K$, $S_k$ is a smooth $\sigma$-approximation of the weak limit $S(T_k)\eqdef S(\cdot,T_k)$ as described in Lemma~\ref{mthm:B}. Moreover, for a blow-up sequence $\set{t_i}\nearrow T_k$,
	\begin{eq}\label{eq:energy_mono_desc_approximate}
		\lim_{t_i\to T_k}\eng(S_{k-1}(\cdot,t_i);\Sigma)\geq\eng(S_k(\cdot,T_k);\Sigma)
		+N(T_k)(\eps_0-\sigma),
	\end{eq}
	where $N(T_k)$ is the number of blowup points at time $T_k$. In particular, the energy of generalized solution is monotonically decreasing and so as $t\to\infty$, one obtains the weak limit $S_\infty$ of $S(\cdot,t)$, which is smooth on $\Sigma$ and solves \eqref{eq:meq}.
\end{mthm}
Note that the piecewise smooth solution given above still depends on a parameter $\sigma$. We expect that if $\sigma$ goes to zero, the piecewise smooth solutions converge to a global weak solution in the sense of Struwe \cite{Struwe1985evolution}*{Prop.~4.2} because $S_k$ is just one particular way of approximating $S(T_k)$ in $W^{1,2}$ norm.

As a byproduct of the above discussion, we obtain another proof of Lemma A.3 of Rivi\`ere \cite{Riviere2007Conservation}. The original proof used the method of continuity under a smallness assumption and a variational proof was obtained by  Schikorra~\cite{Schikorra2010remark}*{Thm.~2.1}, Fr\"ohlich \& M\"uller~\cite{FrohlichMuller2011existence}*{Prop.~2}.
\begin{mcor}[see also {\cite[Lem.~A.3]{Riviere2007Conservation}}]\label{mthm:D}
	Let $E$ be a vector bundle with structure group $\SO(m)$ over disc $B\subset\R^2$, then for every connection 1-form $\Omega\in L^2(B,\so(m)\otimes\wedge^1\R^2)$, one can find $\xi\in W^{1,2}(B,\so(m))$ and $S\in W^{1,2}(B,\SO(m))$ such that
	\begin{eq}\label{eq:riviere}
		\begin{cases}
			\nabla^\perp\xi=S^{-1}\nabla S+S^{-1}\Omega S, & x\in B, \\
			\xi=0, & x\in\pt B.
		\end{cases}
	\end{eq}
	Moreover, there exists some constant $C(m)$, such that
	\[
		\|\xi\|_{1,2;B}+\|\nabla S\|_{2;B}\leq C(m)\|\Omega\|_{2;B}.
	\]
\end{mcor}
The paper is organized as follows. In \autoref{sec:pre}, we present some basic facts and notations, then we show that problem \eqref{eq:mflw} is the negative gradient flow of $\eng(S)$ and that the short time existence follows from the maximum principle and the standard parabolic theory. After that, we show three classical but important lemmas, which are \emph{local energy inequality}, \emph{$\eps$-regularity} and \emph{removable of singularity} in \autoref{sec:key_lemmas}. In \autoref{sec:blw}, we prove that the possible bubbles at singular points are harmonic spheres into the Lie group $G$ and the \emph{boundary bubbles}  do not exist. In \autoref{sec:gen_sol}, an oscillation estimate (Prop.~\ref{prop:osc_estimats}) is proved and applied to the construction of the delicate approximation preserving the boundary condition in Lemma~\ref{mthm:B}, from which the existence of generalized solution (Theorem~\ref{mthm:C}) follows. Finally, in \autoref{sec:app}, we show that Corollary~\ref{mthm:D} can be proved by the heat flow method.
\section{Preliminary}
\label{sec:pre}
In this section, we first review some basic definitions related to connections and bundles. The main purpose is to introduce the notations we used in this paper. Then we prove that \eqref{eq:mflw} is indeed the negative gradient flow of $\eng(S)$ and show that the boundary condition in \eqref{eq:mflw} arises naturally in the first variation of $\eng(S)$. In the last subsection, we prove the local existence of the flow.
\subsection{Notations}
We basically follow the notations of \cite{DonaldsonKronheimer1990geometry, FreedUhlenbeck1991Instantons, Lawson1985gauge}. Let $G$ be a compact Lie group, which is considered as a subgroup $G\subset \SO(r)$ by Peter-Weyl theorem, and $E$ be a real metric vector $G$-bundle of rank $r$ on a Riemannian manifold $M^m$. For an open cover $\set{U_\ga}_{\ga\in\Lambda}$ of $M$, suppose that $\Phi_\ga\mathpunct{:}E|_{U_\ga}\to U_\ga\times\R^r$ is a local trivialization and the transition functions $\phi_{\ga\gb}\mathpunct{:}U_\ga\cap U_\gb\to G$ are defined by
\begin{align*}
	\Phi_\gb\comp\Phi_\ga^{-1}\mathpunct{:}(U_\ga\cap U_\gb)\times\R^r & \to(U_\ga\cap U_\gb)\times\R^r \\
	(x,\xi_\ga) & \mapsto(x, \phi_{\gb\ga}(x)\xi_\ga).
\end{align*}

A \emph{covariant derivative} on $E$ is a first order differential operator
\[
	\nabla\mathpunct{:} \Gamma(E)\to \Gamma(E\otimes T^*M),
\]
which satisfies the Leibnitz rule
\[
	\nabla(f\sigma)=\sigma\otimes\rd f+f\nabla\sigma,\quad\forall f\in C^\infty(M), \sigma\in\Gamma(E),
\]
and is compatible with the metric on $E$
\[
	\rd\inner{\sigma,\tau}=\inner{\nabla\sigma,\tau}+\inner{\sigma,\nabla\tau},\quad\sigma,\tau\in\Gamma(E).
\]

Such a covariant derivative is in one to one correspondence with a \emph{$G$-connection} $A$, which is given (with a fixed trivialization) by a family of $\g$ (Lie algebra of $G$) valued 1-form $A_\ga$ on $U_\ga$ satisfying the following relation on overlaps
\begin{equation*}
	A_\gb=\phi^{-1}_{\ga\gb}d\phi_{\ga\gb}+\phi^{-1}_{\ga\gb}A_\ga\phi_{\ga\gb}.
\end{equation*}
More precisely, given 1-forms $A_\ga$ as above, one can define a covariant derivative by
\[
	(\nabla\sigma)|_{U_\ga}=\rd(\sigma|_{U\ga})+A_\ga\sigma|_{U_\ga}.
\]
We write $\nabla$ as $\nabla_A$ to emphasize its relation to the connection $A$ if necessary and call $\nabla$ a connection.

The difference of any two connections lies in the linear space $\Omega^1(\g_E)\eqdef\Gamma(\g_E\otimes T^*M)$, where $\g_E$ is the associate bundle of $E$ with typical fiber $\g$ and adjoint representation (for the construction of associated bundles, see for example \cite{Michor2008in}*{p.~215, sec.~18.7}). Denote the space of all $G$-connections by $\CA$, which is an affine space.

The covariant derivative $\nabla$ on $E$ can be extended to be an \emph{exterior covariant differential} operator $D$ on $\Omega^k(E)\eqdef\Gamma(E\otimes\wedge^k(T^*M))$ by the rule
\[
	D(\sigma\otimes\omega)=\nabla\sigma\wedge\omega+\sigma\otimes d\omega,\quad\sigma\in\Gamma(E),\omega\in\Omega^k\eqdef\Gamma(\wedge^k(T^*M)).
\]
This extension is independent of the Riemannian metric on $M$, since the exterior differential operator $\rd$ is independent of metric.

The \emph{formal adjoint operator} of $D$ is defined by $D^*(\sigma\otimes\omega)\eqdef(-1)^{mk+1}*D*(\sigma\otimes\omega)$, for $\sigma\in\Gamma(E)$ and $\omega\in\Omega^{k}$, where $*$ is the \emph{Hodge star operator}. The \emph{global inner product} on $\Omega^k(E)$ is defined as
\[
	\bpar{\sigma\otimes\omega,\sigma'\otimes\omega'}
	=\int_M\inner{\sigma\otimes\omega,\sigma'\otimes\omega'}*1
	=\int_M\inner{\sigma,\sigma'}\inner{\omega,\omega'}*1
	=\int_M\inner{\sigma,\sigma'}\omega\wedge*\omega'.
\]
An important relation of $D$ and $D^*$ is the integration by parts formula. We will need the following special case, for $\sigma,\sigma'\in\Gamma(E)$ and $\omega\in\Omega^1$,
\begin{eq}\label{eq:intgrate_by_part}
	\int_{\pt M} \inner{\sigma,\nu\lh(\sigma'\otimes\omega)}i_\nu(*1)
	=\bpar{D\sigma,\sigma'\otimes\omega} -\bpar{\sigma,D^*(\sigma'\otimes\omega)},
\end{eq}
where $\nu$ is the \emph{exterior} unit normal vector on $\pt M$ and $i_\nu(*1)$ is the volume form on $\pt M$ with respect to the induced orientation. Hereafter, we write $D$ as $\nabla$ and $D^*$ as $\nabla^*$.

By the definition of $\nabla^*$, one can show that \eqref{eq:gauge_fixing} follows from \eqref{eq:meq}. In fact,
\begin{eq}\label{eq:conj_star_diff}
	\nabla_{A+a}^*b=\nabla_A^*b+\set{a,b},\quad\forall a,b\in\Omega^1(E),
\end{eq}
where $A$ is a connection on $E$ and the bilinear form $\set{a,b}$ is the tensor product of the Lie bracket on $\g_E$ and the Riemannian metric on $\Omega^1$. In particular, $\set{a,a}=0$.

Denote the fiber bundle of gauge transformations by $\Aut_G E$, which is the associated bundle of $E$ with typical fiber $G$ and the conjugate action. The gauge group is the group formed by all the sections of $\Aut_GE$. There is a pointwise exponential map $\exp\mathpunct{:}\Gamma(\g_E)\to\Gamma(\Aut_GE)$. $\g_E$ and $\Aut_GE$ are metric bundles with metric induced from $E$. There is also an induced $G$-connection on $\g_E$ and $\Aut_GE$ given by the rule
\[
	(\nabla T)(\sigma)=\nabla(T(\sigma))-(-1)^kT\wedge \nabla \sigma,\quad\sigma\in\Gamma(E),T\in\Omega^k(\End E),
\]
where we regard $S\in\Gamma(\Aut_GE)$ as a section of the vector bundle $\End E$. $\g_E$ is the infinitesimal automorphism bundle of $\Aut_GE$, which is preserved by the covariant differential operator $\nabla$ on $\End E$ \cite{Lawson1985gauge}*{p.~21ff}.

A gauge transformation $S\in\Gamma(\Aut_GE)$ acts on a connection $A$ by ``pull back'':
\[
	\nabla_{S^*(A)}(\cdot)=S^{-1}\comp\nabla_A(S(\cdot)),
\]
i.e.,
\begin{eq}\label{eq:pullback}
	S^*(A)=A+S^{-1}\nabla_AS.
\end{eq}

We will need the local form of \eqref{eq:meq} to do analysis. In local computation, one can either use moving frame or coordinate basis, and the Einstein's summation convention is always assumed. Let $\set{e_i}$ be an orthonormal local frame of $M$ and $\set{\omega^j}$ be the dual co-frame. Recall, by the definition of $\nabla^*$, one can show that $\nabla^*a=-(\nabla_{e_j}a)(e_j)$ for any $a\in\Omega^1(\g_E)$ (see, for example \cite{Lawson1985gauge}*{p.~21, (3.6)}). Note also the relation given by \eqref{eq:conj_star_diff}, it's easy to rewrite \eqref{eq:meq} as
\begin{eq}\label{eq:meq_global}
	-\nabla_{A}^*\nabla_{A}S=\nabla_{A;e_j}SS^{-1}\nabla_{A;e_j}S +S\set{S^{-1}\nabla_{A}S,a}+S\nabla_{A}^*a,
\end{eq}%
where $a=A-A_0\in\Omega^1(\g_E)$.

To write down \eqref{eq:meq_global} exactly in local coordinate systems, suppose $\set{\mu_\ga}$ is a local frame of $E$ and $\pt_i$ is a local coordinate basis of $M$, then locally,
\[
	\nabla_{A; \pt_i}=\pt_i+A_i,\quad\text{or equivalently }\nabla_A=\rd+A,
\]
where $A_i$ are matrices given by $\nabla_{A;\pt_i}\mu_\gb\eqdef A_{i\gb}^\ga\mu_\ga$. With the help of $\set{\mu_\ga}$, $S|_U$ becomes a matrix-valued function $u$ with $u(x)\in G\subset \SO(r)\subset \R^{r\times r}$. We introduce the notation $u_{|k}$ as the local expression of $\xkf{\nabla_{A;\pt_k}S}|_U$
\[
	u_{|k}\eqdef\pt_ku+[A_k,u].
\]
Recall that for a function $f$ on $(M,g)$, the Hessian and Laplace of $f$ are given by
\[
	(\nabla^2 f)_{kl}=\pt_k\pt_lf-\Gamma_{kl}^i\pt_if\text{ and }
	\Delta_g f=\tr_g(\nabla^2f)=g^{kl}(\nabla^2f)_{kl},
\]
respectively. Now, we expand \eqref{eq:meq_global} term by term. Firstly,
\begin{align*}
	-\nabla_A^*\nabla_A S & =-\nabla_A^*\left(\nabla_{A;\pt_i}S\otimes\rd x^i\right) \\
	& =g^{kl}\left\{\nabla_{A;\pt_k}\left(\nabla_{A;\pt_i}S\otimes\rd x^i\right)\right\}(\pt_l) \\
	& =g^{kl}\left\{\nabla_{A;\pt_k}\nabla_{A;\pt_l}S-\Gamma_{kl}^i\nabla_{A;\pt_i}S\right\},\\
\end{align*}
thus the local expression for $-\nabla_A^*\nabla_AS$ is
\begin{multline*}
	g^{kl}\left\{\nabla_{A;\pt_k}\left(\pt_l u+[A_l,u]\right)-\Gamma_{kl}^i\left( \pt_i u+[A_i,u]\right)
\right\}\\
=\Delta_g u+g^{kl}\left\{\pt_k[A_l,u]+[A_k,u_{|l}]-\Gamma_{kl}^i[A_i,u]\right\}.
\end{multline*}
Clearly, the local expression for $g^{kl}\nabla_{A;\pt_k}SS^{-1}\nabla_{A;\pt_l}S$ is $g^{kl}u_{|k}u^{-1}u_{|l}$. By the definition of the bilinear form $\set{\cdot,\cdot}$ ( see \eqref{eq:conj_star_diff}) and the adjoint operator $\nabla^*$, one can show that the local expression of $\set{S^{-1}\nabla_AS,a}$ is $-g^{kl}\left[u^{-1}u_{|k},a_l\right]$
and
\[
	\nabla_{A}^*a=\rd^*a-g^{kl}[A_k,a_l]
	=-g^{kl}\left(\pt_la_k-\Gamma_{kl}^ia_i+[A_k,a_l]\right)
	=-g^{kl}\xkf{a_{l|k}-a_i\Gamma_{kl}^i}.
\]
In particular, when the metric $g$ of $M$ is Euclidean, the local expression of \eqref{eq:meq} is
\begin{eq}\label{eq:loc_u}
	\Delta u- u_{|k}u^{-1}u_{|k} =-\pt_k[A_k,u]-[A_k,u_{|k}]-u[u^{-1}u_{|k},a_k]-ua_{k|k},
\end{eq}
with boundary condition
\begin{eq}\label{eq:loc_bdy}
	\nu\lh\xkf{u^{-1}\xkf{du+[A,u]}+a}=0.
\end{eq}

It is clear from \eqref{eq:loc_u} that the leading term is the same as the tension field of harmonic map into $G$. In fact, it easy to verify $(u^{-1}u_{|k})^t=-u^{-1}u_{|k}$, thus $u_{|k}\in T_uG=u\g$. To compute the second fundamental form of $G\embedto\R^{r\times r}$, we firstly note that since any Lie subgroup of a compact Lie group is totally geodesic, one can take $G=SO(r)$ in the following computation. Clearly, the Lie algebra of $G$ is given by $\g=\set{B\in M_{r\times r}|B=-B^t}$, the tangent space of $G$ at $u$ is $T_{u}G=u\g=\set{uB|B\in\g}$ and the orthocomplement of $T_uG$ is $T_{u}^{\perp}G=\set{uC|C=C^t,\,C\in M_{r\times r}}$, where the inner product of $M_{r\times r}$ is given by $\inner{B,C}=\tr(BC^t)$. For a tangent vector $X=uB$ at $u$, clearly $\gamma(t)\eqdef u\exp(tB)$ is a curve tangent to $X$ and by the definition of $\II$, we have
\[
	\II(X,X)=[\ddot\gamma(t)]^\perp=[uB^2]^\perp=u\frac{B^2+(B^2)^t}{2}=uB^2=Xu^{-1}X.
\]
\subsection{The Negative Gradient Heat Flow}
We will show that the heat flow \eqref{eq:mflw} is the negative gradient flow of the energy of $S$,
\[
	\eng(S)\eqdef\frac{1}{2}\int_\Sigma|S^*(A)-A_0|^2.
\]
\begin{prop}[Monotonicity along the flow]\label{prop:evolution_of_energy}
	The oblique initial-boundary value problem \eqref{eq:mflw} is a negative gradient heat flow of $\eng(S)$.
\end{prop}
\begin{proof}
	Let $S(t)=S\exp(t\xi)$ be a vertical variation of $S$, where $\xi\in\Omega^0(\g_E)$, then for any $\sigma\in\Omega^0(E)$, by the definition of induced connection,
	\begin{align*}
		\ddt{0}(S^*(A)-A_0)(\sigma)
		& =\ddt{0}\left[S^{-1}(t)\nabla_{A}S(t)+A-A_0\right](\sigma) \\
		& =\ddt{0}S^{-1}(t)\nabla_{A}(S(t)\sigma) \\
		& =-\xi S^{-1}\nabla_{A}(S\sigma)+S^{-1}\nabla_{A}(S\xi \sigma) \\
		& =(\nabla_{A}\xi)(\sigma)+\left[S^{-1}\nabla_{A}S,\xi\right](\sigma) \\
		& =(\nabla_{S^*(A)}\xi)(\sigma).
	\end{align*}
	Therefore,
	\begin{align*}
		\left.\ppt{t}\right|_{t=0}\eng(S(t))
		& =\int_\Sigma\inner{\left.\ppt{t}\right|_{t=0}(S(t)^*(A)-A_0),S^*(A)-A_0} \\
		& =\int_\Sigma\inner{\nabla_{S^*(A)}\xi,S^*(A)-A_0}, \\
		\intertext{apply the integration by parts formula \eqref{eq:intgrate_by_part} and the boundary condition}
		RHS& =\int_{\pt \Sigma}\inner{\xi,\nu\lh(S^*(A)-A_0)}i_\nu(*1)+(\xi,\nabla_{S^*(A)}^*(S^*(A)-A_0)) \\
		& =(\xi,\nabla_{S^*(A)}^*(S^*(A)-A_0)).
	\end{align*}
	This implies that \eqref{eq:mflw} is a negative gradient flow of $\eng(S)$, and the energy is monotonically non-increasing along the flow.
\end{proof}
\subsection{Local Existence of the Flow}\label{sec:local_existence}
The local expression \eqref{eq:loc_u} implies that \eqref{eq:mflw} is a strict parabolic system, and the oblique initial-boundary condition of \eqref{eq:mflw} is compatible by assumption, thus the standard theory of parabolic systems (an imitation of \cite{Lieberman1996Second}*{Thm.~8.12}) implies that there exists some $S\in C^{2,\alpha}(\Sigma\times[0,\eps),\End E)\cap C^\infty(\Sigma\times(0,\eps),\End E)$, for some $\alpha\in(0,1)$ and $\eps$ small, which solves the system \eqref{eq:mflw}. However, one still needs to show that the solution stays on $\Aut_GE$. %]

The same issue arises in the study of harmonic map flow and there are two ways to settle it. Hamilton \cite{Hamilton1975Harmonic}*{p.~108ff} proved a uniqueness theorem and then embedded the target manifold $Y$ properly into a Euclidean space with non-flat metric, such that there is a reflection on a tabular neighborhood of $Y$, which would give two solutions with the same initial-boundary condition as the original one, and concluded that the solution must lie in $Y$. On the other hand, Eells \& Sampson \cite{EellsSampson1964Harmonic}*{Sec.~7 (C), Theorem} considered the evolution of distance function from the solution to the target manifold. Since $[\Delta_g u+A(u)(\nabla u,\nabla u)]^\perp=[u_t]^\perp=0$ for any $u$ on the target manifold, the normal part vanishing property holds for harmonic heat flow. It enables one to apply the maximum principle to show that the solution stays on the target manifold along the flow. In the following, we will generalize the latter method.
\begin{prop}\label{prop:stay_on_mfd}
	Suppose that $\B$ is a metric vector bundle over compact manifold $M$ with boundary, $\E$ is a sub-fiber bundle of $\B$ with compact fibers. For a section $S$ of $\B$, consider the following heat flow
	\begin{eq}\label{eq:general_flw}
		\begin{cases}
			\heatop S=f(\nabla S,S), & (x,t)\in M\times(0,T), \\
			S(\cdot,0)=S_0\in\Omega^0(\E), & x\in M, \\
			\nu\lh(S^*(A)-A_0)=0, & (x,t)\in\pt M\times[0,T), %]
		\end{cases}
	\end{eq}
	where $f:\Omega^1(\B)\times\Omega^0(\B)\to\Omega^0(\B)$ is a smooth bundle map, and $\Delta=-\nabla^*\nabla$ is the (negative) rough Laplace, $\nabla=\nabla_A$ is a metric compatible connection. If there holds the \emph{normal part vanishing property}
	\[
		\left[\Delta \tilde S+f(\nabla\tilde S,\tilde S)\right]^\perp=0,\quad\forall \tilde S\in\Omega^0(\E),
	\]
	then any \emph{smooth} solution of \eqref{eq:general_flw} must lie in $\E$.
\end{prop}
Before going into the proof, we remark that the assertion (\ref{lab:1_mthm:A}) of Theorem~\ref{mthm:A} follows immediately from Prop.~\ref{prop:stay_on_mfd}. In fact, take $\B$ be the endomorphism bundle $\End E$ and $\E$ be the gauge bundle $\Aut_G E$, the normal part vanishing property holds, since
\[
	\nabla^*_{\tilde S^*(A)}(\tilde S^*(A)-A_0)\in\Omega^0(\g_E),\quad\forall\tilde S\in\Omega^0(\Aut_G E).
\]
\begin{proof}[proof of Prop.~\ref{prop:stay_on_mfd}]
	Let $h$ be the metric on $\B$, therefore $h_x\eqdef\inner{\cdot,\cdot}_x$ is an inner product on each fiber $\B_x$ which is smooth with respect to $x$. Since each fiber $\E_x$ is compact and smooth, it is well known that there exists $\delta=\delta(x)$, such that for any $v\in \B_x$, with $d(v,\E_x)\eqdef\inf_{w\in\E_x}[h_x(v-w,v-w)]^{1/2}< \delta$, there exists a unique $\bar v\in\E_x$, satisfying
	\[
		d(v,\E_x)=[h_x(v-\bar v,v-\bar v)]^{1/2}.
	\]
	Thus, the fiber-wise projection map
	\[
		\pi_x:\B_x\to\E_x,\quad \pi_x(v)=\bar v,
	\]
	is well defined for $v$ satisfying $d(v,\E_x)< \delta$. Moreover, it is clear that $v-\pi_x(v)\perp\E_x$. Given a smooth section $V$ of $\B$ with $V(x)$ belonging to the $\delta$-neighborhood of $\E_x$, the assignment $x\mapsto\pi_x(V(x))$ leads to a smooth section of $\E$, denote it by $\pi(V)$.

	Now, for $t\in[0,\eps)$, $\eps>0$ is sufficiently small, $S(x,t)$ belongs to the $\delta$-neighborhood of $\E_x$. Consider the distance function of $S$%]
	\[
		\rho(x,t)\eqdef\frac{1}{2}\inner{S(x,t) -\pi_x(S(x,t)),S(x,t)-\pi_x(S(x,t))}_x.
	\]
	In order to simplify the notation, we just write
	\[
		\rho(x,t)\eqdef\frac{1}{2}\inner{S-\pi(S),S-\pi(S)}_x.
	\]
	Then, the evolution of $\rho$ is given by
	\begin{align*}
		\heatop\rho
		& =\inner{S-\pi(S),\pt_t S-\rd\pi_x|_{S(x)}(\pt_{t}S)}_x+\rd^*\rd\rho \\
		& =\inner{S-\pi(S),\pt_t S}_x+\rd^*\inner{S-\pi(S),\nabla(S-\pi(S))}_x \\
		& =\inner{S-\pi(S),\pt_tS-\nabla_{e_j}\nabla_{e_j}(S-\pi(S))}_x \\
		& \qquad+\omega^j(\nabla_{e_i}e_i)\inner{S-\pi(S),\nabla_{e_j}(S-\pi(S))}_x \\
		& \qquad-\inner{\nabla_{e_j}(S-\pi(S)),\nabla_{e_j}(S-\pi(S))}_x \\
		& =\inner{S-\pi(S),\pt_t S+\nabla^*\nabla(S-\pi(S))}_x-\inner{\nabla_{e_j}(S-\pi(S)),\nabla_{e_j}(S-\pi(S))}_x \\
		& =\inner{S-\pi(S),\pt_tS-\Delta S+\Delta(\pi(S))}_x-|\nabla_{e_j}(S-\pi(S))|^2 \\
		& =\inner{S-\pi(S),f(\nabla S,S)+\Delta(\pi(S))}_x-|\nabla_{e_j}(S-\pi(S))|^2.
	\end{align*}
	Thus,
	\begin{align*}
		\heatop\rho+|\nabla_{e_j}(S-\pi(S))|^2
		& =\inner{S-\pi(S),f(\nabla S,S)+\Delta(\pi(S))}_x \\
		& =\big\langle S-\pi(S),f(\nabla S,S)-f(\nabla(\pi(S)),S) \\
		& \qquad+f(\nabla(\pi(S)),S)-f(\nabla(\pi(S)),\pi(S)) \\
		& \qquad+f(\nabla(\pi(S)),\pi(S))+\Delta(\pi(S))\big\rangle_x.
	\end{align*}
	Let us process the three terms one by one. Set $S_\tau=\pi(S)+\tau(S-\pi(S))$, $\tau\in[0,1]$, then
	\begin{align*}
		\inner{S-\pi(S),f(\nabla S,S)-f(\nabla(\pi(S)),S)}_x
		& =\inner{S-\pi(S),\int_0^1\pt_\tau f(\nabla S_\tau,S)\rd \tau}_x \\
		& =\int_0^1\inner{S-\pi(S),\rd f|_{(\nabla S_\tau ,S)}(\nabla(S-\pi(S)),0)}_x\rd \tau \\
		& \leq \|\rd f\|\cdot |S-\pi(S)|\cdot|\nabla(S-\pi(S))|\\
		& =\sqrt{2\rho}\|\rd f\|\cdot|\nabla_{e_j}(S-\pi(S))|\\
		&\leq \rho\|df\|^2+\frac{1}{2}|\nabla_{e_j}(S-\pi(S))|^2.
	\end{align*}
	In the same manner,
	\[
		\inner{S-\pi(S),f(\nabla(\pi(S)),S)-f(\nabla(\pi(S)),\pi(S))}_x\leq2\rho\|\rd f\|.
	\]
	By the normal part vanishing property, the last term vanishes. In conclusion,
	\begin{eq}\label{eq:dist}
		\heatop\rho\leq \rho\|\rd f\|^2+2\rho\|df\|\leq C(f)\rho.
	\end{eq}
	Now, on the parabolic boundary, one has
	\begin{align*}
		\rho & =0,\quad\text{on } M\times\set{t=0}, \\
		\nu\lh\rd\rho & =\nu\lh\inner{\nabla_A(S-\pi(S)),S-\pi(S)} \\
		& =\nu\lh\inner{\nabla_AS,S-\pi(S)} \\
		& =\nu\lh\inner{-Sa,S-\pi(S)} \\
		& =-\nu\lh\inner{(S-\pi(S))a+\pi(S)a,S-\pi(S)} \\
		& =-\nu\lh\inner{(S-\pi(S))a,S-\pi(S)}.
	\end{align*}
	The last term vanishes because
	\[
		\inner{(S-\pi(S))a,S-\pi(S)}=\frac{1}{2}\ddt{0}|(S-\pi(S))\exp(ta)|^2,
	\]
	and the action of $\exp(ta(x))\in G\embedto SO(r)\subset M_{r\times r}$ on $S(x,t)-\pi(S(x,t))\in M_{r\times r}$ is an isometry.

	Since $\rho$ satisfies \eqref{eq:dist} with the boundary condition $\CM(\rho)=0$, where
	\[
		\CM(\rho)\eqdef
		\begin{cases}
			\rho, & \bar M\times\set{t=0} \\
			\nu\lh\rd\rho, & \pt M\times(0,T),
		\end{cases}
	\]
	the Proposition follows from maximum principle.
\end{proof}
\section{Key Estimates \& Important Lemmas}\label{sec:key_lemmas}
In the forthcoming subsections, we prove some basic estimates needed in the blowup analysis. The first one is the local energy inequality, which allows us to compare the energy at different times. The second one is the $\eps$-regularity, which admits local regularity of solutions under some smallness of energy condition. The last one is removable of singularity, which asserts that isolated singularity (maybe at boundary) of our equations is removable.
\subsection{Local Energy Inequality}
In the case of harmonic map flow, Struwe presented a local energy inequality in \cite{Struwe1985evolution}*{Lem.~3.6}, which asserts that the formation of bubble needs time when energy concentrates at a point. Since our boundary condition is good enough, following the method of Struwe, one can show the following version for \eqref{eq:mflw}, which holds both in the interior and at the boundary.
\begin{prop}[Local Energy Inequality]\label{prop:local_eng_ineq}
	Suppose $S\in\Gamma(\Aut_GE)$ is a solution of \eqref{eq:mflw} with initial energy $\eng(S_0;\Sigma)< +\infty$ ($S_0$ is smooth). Then for $x_0\in \Sigma$, $0< T_1< T_2\leq T$ and $0< R_2< R_1\leq i_\Sigma$, where $i_\Sigma$ is the injectivity radius of $\Sigma$, there holds the \emph{local energy inequality}
	\begin{align*}
		\eng(S(T_2);D_{R_2}(x_0)) & \leq \eng(S(T_1);D_{R_1}(x_0))+C\frac{|T_2-T_1|}{|R_1-R_2|^2}\eng(S_0;\Sigma), \\
		\intertext{and the \emph{reverse local energy inequality},}
		\eng(S(T_1);D_{R_2}(x_0)) & \leq\eng(S(T_2);D_{R_1}(x_0)) +2\int_{T_1}^{T_2}\int_{D_{R_1}(x_0)}|\pt_tS|^2 +C\frac{|T_2-T_1|}{|R_1-R_2|^2}\eng(S_0;\Sigma),
	\end{align*}
	for some universal constant $C>0$, where $D_r(x_0)\eqdef
	\overline{B_r(x_0)}\cap\Sigma$.
\end{prop}
\begin{rmk}
	We will need the boundary version because one cannot exclude energy concentrates at the boundary. For example, to show the finiteness of singularity points on $\Sigma$ at a singular time, one needs to apply the local energy inequality at the boundary.
\end{rmk}
\begin{proof}
	Let $\Sigma_{T_1}^{T_2}\eqdef\set{(x,t)|x\in \Sigma, t\in[T_1,T_2]}$ and $\phi\in C^\infty_0(D_{R_1}(x_0))$ be a cut off function with $\phi\equiv 1$ on $D_{R_2}(x_0)$, $|\nabla\phi|\leq\frac{2}{R_1-R_2}$ and $0\leq\phi\leq1$. Firstly, by the boundary condition, one can compute
	\begin{align*}
		\ppt{t}\int_\Sigma e(s)\phi^2 & =\int_\Sigma\inner{S^*(A)-A_0,\pt_t(S^{-1}\nabla_AS)}\phi^2 \\
		& =\int_\Sigma\inner{S(S^*(A)-A_0),-\pt_tSS^{-1}\nabla_AS+\nabla_A\pt_tS}\phi^2 \\
		& =\int_\Sigma\inner{S(S^*(A)-A_0)(e_i)S^{-1}\nabla_{A;e_i}S,\pt_tS}\phi^2
		+\inner{\nabla_A^*(S(S^*(A)-A_0)\phi^2),\pt_tS}\\
		& \qquad+\int_{\pt \Sigma}\inner{\nu\lh(S(S^*(A)-A_0)\phi^2),\pt_tS}i_\nu(*1) \\
		& =\int_\Sigma\inner{S(S^*(A)-A_0)(e_i)\comp S^{-1}\nabla_{A;e_i}S,\pt_tS}\phi^2
		+\inner{\nabla_A^*(S(S^*(A)-A_0)\phi^2),\pt_tS}\\
		& \overset{\eqref{eq:conj_star_diff}}{=}\int_\Sigma\inner{\nabla_{A+S^{-1}\nabla_AS}^*(S(S^*(A)-A_0)\phi^2)+S^{-1}\nabla_{A;e_i}S\comp S(S^*(A)-A_0)(e_i)\phi^2,\pt_tS} \\
		& =\int_\Sigma\inner{\nabla_{S^*(A)}^*(S(S^*(A)-A_0)\phi^2)+\nabla_{S^*(A);e_i}S\comp(S^*(A)-A_0)(e_i)\phi^2,\pt_tS},
	\end{align*}
	the last equality follows from the fact
	\[
		\nabla_{A+S^{-1}\nabla_AS;e_i}S=\nabla_{A;e_i}S+[S^{-1}\nabla_{A;e_i}S,S]
		=S^{-1}\nabla_{A;e_i}S\comp S.
	\]
	Now, note that for any $T\in\Omega^1(\g_E)$, there holds
	\begin{eq}\label{eq:loc_eng_ineq_compute_1}
		\nabla_A^*(S\comp T)=S\nabla_A^*T-\nabla_{A;e_i}S\comp T(e_i),
	\end{eq}
	and
	\begin{eq}\label{eq:loc_eng_ineq_compute_2}
		\nabla_A^*(T\phi^2)=\phi^2\nabla_A^*T-d(\phi^2)(e_i)T(e_i).
	\end{eq}
	Apply the flow equation \eqref{eq:mflw}, we continue the computation as
	\begin{align*}
		& \nabla_{S^*(A)}^*(S(S^*(A)-A_0)\phi^2)+\nabla_{S^*(A);e_i}S\comp(S^*(A)-A_0)(e_i)\phi^2 \\
		& \qquad\overset{\eqref{eq:loc_eng_ineq_compute_1}}{=}S\nabla_{S^*(A)}^*[(S^*(A)-A_0)\phi^2] \\
		& \qquad\overset{\eqref{eq:loc_eng_ineq_compute_2}}{=}S\nabla_{S^*(A)}^*(S^*(A)-A_0)\phi^2-d(\phi^2)(e_i)S\comp(S^*(A)-A_0)(e_i) \\
		& \qquad\overset{\eqref{eq:mflw}}{=}-\pt_tS\phi^2-d(\phi^2)(e_i)S\comp(S^*(A)-A_0)(e_i).
	\end{align*}
	Therefore, by Cauchy-Schwarz inequality,
	\begin{align*}
		\int_{T_1}^{T_2}\ppt{t}\int_\Sigma e(S)\phi^2+\int_{\Sigma_{T_1}^{T_2}}|\pt_tS|^2\phi^2
		&=-\int_{\Sigma_{T_1}^{T_2}}2\phi\rd\phi(e_i)\inner{S\comp(S^*(A)-A_0)(e_i),\pt_tS}\\
		& \leq\int_{T_1}^{T_2}\int_{\Sigma}2|\nabla\phi||S^*(A)-A_0|\cdot\phi|\pt_tS| \\
		& \leq2\int_{T_1}^{T_2}\int_\Sigma|\nabla\phi|^2e(S)+\int_{\Sigma_{T_1}^{T_2}}|\pt_tS|^2\phi^2 \\
		& \leq\int_{T_1}^{T_2}\frac{8}{|R_1-R_2|^2}\eng(S(\cdot,t);D_{R_1})+\int_{\Sigma_{T_1}^{T_2}}|\pt_tS|^2\phi^2 \\
		& \leq\frac{8|T_2-T_1|}{|R_1-R_2|^2}\eng(S(\cdot,0))+\int_{\Sigma_{T_1}^{T_2}}|\pt_tS|^2\phi^2,
	\end{align*}
	i.e.,
	\begin{align*}
		8\frac{|T_2-T_1|}{|R_1-R_2|^2}\eng(S_0) & \geq\int_\Sigma e(S(\cdot,T_2))\phi^2-\int_\Sigma e(S(\cdot,T_1))\phi^2 \\
		& \geq\eng(S(T_2);D_{R_2})-\eng(S(T_1);D_{R_1}).
	\end{align*}
	Similarly,
	\begin{align*}
		-\frac{8|T_2-T_1|}{|R_1-R_2|^2}\eng(S_0) & \leq\int_{T_1}^{T_2}\ppt{t}\int_\Sigma e(S)\phi^2+2\int_{\Sigma_{T_1}^{T_2}}|\pt_tS|^2\phi^2 \\
		& \leq\eng(S(T_2);D_{R_1})-\eng(S(T_1);D_{R_2})+2\int_{T_1}^{T_2}\int_{D_{R_1}}|\pt_tS|^2.
	\end{align*}
\end{proof}
\subsection{\epst-Regularity}
The $\eps$-regularity is a key lemma in semi-linear partial differential equations, which asserts that when the energy is small, the critically nonlinear equation behaves like a linear one. It was first discovered for harmonic maps by Sacks and Uhlenbeck in their celebrated paper \cite{SacksUhlenbeck1981existence}*{Main Estimate~3.2}, then Schoen gave a different proof based on an argument by contradiction \cite{Schoen1984harmonic}*{Thm.~2.2}. For the harmonic map flow, we refer to \cite{Struwe1985evolution}*{Lem.~3.10} for a closed surface and \cite{Chang1989Heat}*{Lem.~4.2} for a surface with boundary. Since \eqref{eq:mflw} is a small perturbation of harmonic maps, it can be expected that when the scale is small enough, the $\eps$-regularity still holds. The proof is based on parabolic estimates and an interpolation inequality of Nirenberg.
\begin{thm}[$\eps$-Regularity]\label{thm:epsreg}
	Suppose $S$ is a solution of \eqref{eq:mflw} on $\Sigma\times[0,T)$. Let $z_0=(x_0,t_0)\in \Sigma\times[\delta,T)$, for some fixed $\delta>0$ and denote the parabolic ball by %]
	\[
		P_r(z_0)\eqdef \set{z=(x,t)|x\in\overline{B_r(x_0)}\cap \Sigma=\mathpunct{:}D_r(x_0),\sqrt{t_0-t}\leq r}.
	\]
	Then there exist some $\eps_0>0$ (which is independent to $S$ and $z_0$) and $r_0$ with $0< r_0< \sqrt\delta/2$, such that for $r< r_0$, if there holds the ``small energy'' condition
	\[
		\sup_{[t_0-r^2,t_0)}\eng(S(t);D_r(x_0))\leq\eps_0, %]
	\]
	then
	\[
		\sup_{P_{r/2}(z_0)}|\nabla_A^kS|\leq C_kr^{-k},\quad\forall k=1,2,\cdots,
	\]
	for some constant $C_k$.
\end{thm}
\begin{proof}
	For simplicity, suppose that the metric of $\Sigma$ is Euclidean. On $P_r(z_0)$, $S$ can be viewed as a map (after taking a fixed local trivialization) into $G\embedto SO(r)\subset \R^{r\times r}$, and the problem reduces to \eqref{eq:loc_u}, i.e.,
	\begin{eq}\label{eq:eps_eng:loc_meq}
		\begin{split}
			\heatop{u} &   =-(\pt_ku+[A_k,u])u^{-1}(\pt_ku+[A_k,u])+\pt_k[A_k,u]+[A_k,\pt_ku+[A_k,u]] \\ &\qquad+u[u^{-1}(\pt_ku+[A_k,u]),a_k]+u(\pt_ka_k+[A_k,a_k]).
		\end{split}
	\end{eq}

	For the boundary condition on $P_r(z_0)$, it can be divided into interior case and boundary case. In fact, when $D_r^\circ(x_0)\cap\pt\Sigma\neq\emptyset$, without loss of generality, we can assume the boundary $D_r^\circ(x_0)\cap\pt\Sigma$ is flat. Then there are three cases based on the relative location of $x_0$:
	\begin{itemize}
		\item Interior: $D^\circ_r(x_0)=B_r(x_0)$, that is, the open disc does not touch the boundary; 
		\item Faraway from the boundary: $D^\circ_r(x_0)\cap\pt\Sigma\neq\emptyset$ and $\dist(x_0,\pt\Sigma)\geq r/2$. This case can be reduced to the interior case with $r$ replaced by $r/2$.
		\item Near the boundary: $D^\circ_r(x_0)\cap\pt\Sigma\neq\emptyset$ and $\dist(x_0,\pt\Sigma)< r/2$. In this case, we may assume, without loss of generality, $D_{r}(x_0)$ is the upper half disc $B_{r/2}^+(x_0)$, which is centered at $x_0$ with radius $r/2$. 
	\end{itemize}
	\begin{figure}[!htbp]
		\centering
		\includegraphics{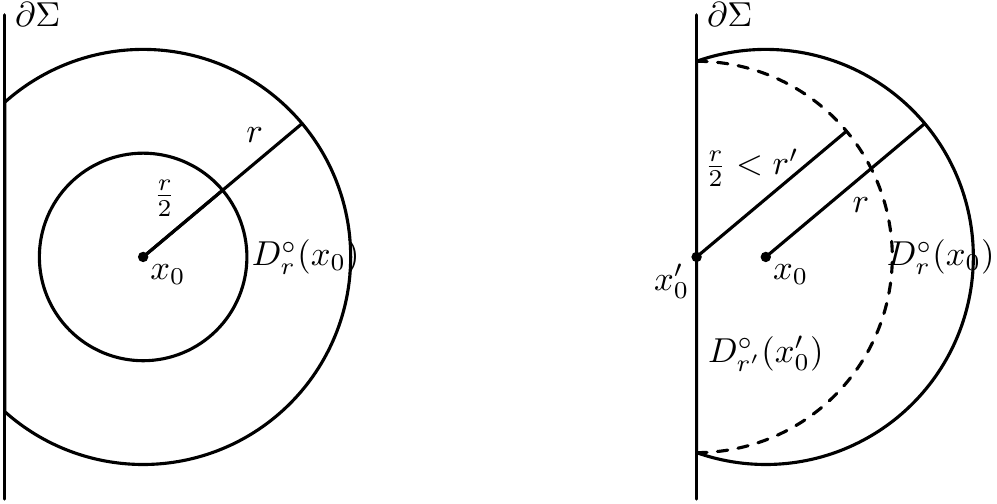}
		\caption{Faraway from the boundary and near the boundary cases}
	\end{figure}
	Thus, in what follows, we will only consider the following two cases:
	\begin{itemize}
		\item Interior: $D_r^\circ(x_0)=B_r(x_0)$, the boundary condition will be of Dirichlet type;
		\item Boundary: $D_r^\circ(x_0)=B^+_{r}(x_0)$, the boundary condition will be oblique.
	\end{itemize}

	Before we turn to the derivation of boundary condition, we note that the problem is conformally invariant. Let $(x,t)\to(x_0+rx,t_0+r^2t)$, which takes $D_r(x_0)\times[t_0-r^2,t_0)$ to $D_1\times[-1,0)$. Then we may assume that $r=1$, $z_0=(0,0)$, with the following smallness condition %]
	\[
		\sup_{t\in[-1,0)}\int_{D_1}|\nabla u|^2\leq 2\eps_0.%]
	\]
	In fact, the smallness condition is
	\[ 
		\sup_{t\in[t_0-r^2,t_0)}\eng(S(t),D_r(x_0))=\sup_{t\in[t_0-r^2,t_0)}\int_{D_r(x_0)}|du+[A,u]+ua|^2\leq\eps_0.
	\]
	For $u\in G\embedto\SO(r)$,
	\[
		\sup_{t\in[t_0-r^2,t_0)}\int_{D_r(x_0)}|[A,u]+ua|^2\leq C\sup_{t\in[t_0-r^2,t_0)}\int_{D_r(x_0)}|u|^2\leq Cr_0^2,%]
	\]
	Thus one can take $r_0$ small enough such that
	\begin{eq}\label{eq:pure_eng_small}
		\sup_{t\in[t_0-r^2,t_0)}\int_{ D_r(x_0)}|\nabla u|^2\leq 2\eps_0. %]
	\end{eq}

	Note also that $u\in G$, and $A_k$, $\pt_kA_k$, $a_k$, $\pt_ka_k$ are all bounded, we can rewrite the scaled equation of \eqref{eq:eps_eng:loc_meq} as
	\begin{eq}\label{eq:loc_form_u}
		\heatop{u}=-\II(u)(\nabla u,\nabla u)+\ga\nabla u+\beta,\quad (x,t)\in P_1\eqdef D_1\times[-1,0), %]
	\end{eq}
	where $\ga,\gb\in L^\infty$, $\II(u)$ is the second fundamental form of $G\embedto \R^{r\times r}$ at $u$ and $\nabla u$ is the gradient of $u$.

	Consider a cutoff function $\phi(x,t)$ such that
	\[
		\phi(x,t)=\begin{cases}
			1,&(x,t)\in  P_{1/2}\\
			0,&(x,t)\not\in  P_{1}.
		\end{cases}
	\]
	Let $v=\phi u$, by \eqref{eq:loc_form_u}, the equation of $v$ is
	\[
		\heatop{v}=-\phi\II(u)(\nabla u,\nabla u)+(\ga\phi-2\nabla\phi)\nabla u+\phi \beta+u\heatop{\phi},
	\]
	which is still of the form
	\begin{eq}\label{eq:eps_eng:loc_form_u_cutoff}
		\heatop{v}=-\phi\II(u)(\nabla u,\nabla u)+\ga\nabla u+\beta,\quad (x,t)\in P_1.
	\end{eq}

	Next, we will derive the boundary condition in both the interior and the boundary cases. It is clear that in the interior case, the initial-boundary condition is Dirichlet type
	\begin{eq}\label{eq:eps_eng:loc_bdr_interior}
		\begin{cases}
			v(x,-1)=0,&x\in B_1\\
			v(x,t)=0,&(x,t)\in\pt B_1\times[-1,0].
		\end{cases}
	\end{eq}
	While in the boundary case, by \eqref{eq:loc_bdy}, it is the following mixed type
	\[
		\begin{cases}
			v(x,-1)=0, & x\in  D_1 \\
			\phi\frac{\pt v}{\pt \nu}=-\nu\lh([A,v]+va)\phi+\frac{\pt\phi}{\pt\nu}v, & (x,t)\in H_1\times [-1,0] \\
			v(x,t)=0, & (x,t)\in(\pt D_1\cap \Sigma)\times[-1,0],
		\end{cases}
	\]
	where $H_1=\overline{B_{1}^+(0)}\cap\pt\Sigma$.

	To proceed the bad term $\frac{\pt\phi}{\pt\nu}v$, we modify $\phi$ such that $\frac{\pt\phi}{\pt\nu}\equiv0$ on $H_1\times[-1,0]$. In fact, it is easy to construct a cutoff function $\phi_1$ on the half ball $B_1^+$, such that
	\[
		\phi_1(r)=\begin{cases}
			1,&r=|x|< 1/2,\\
			0,&r=|x|>1.
		\end{cases}
	\]
	Then we can take $\phi=\phi_1(r)\phi_2(t)$, where $\phi_2(t)$ is a cutoff function
	\[
		\phi_1(t)=\begin{cases}
			1,&t\in[-1/2,0]\\
			0,&t< -1.
		\end{cases}
	\]
	Thus in the boundary case, the initial boundary condition can be written as
	\begin{eq}\label{eq:eps_eng:loc_bdr_boundary}
		\begin{cases}
			v(x,-1)=0, & x\in D_1 \\
			\vppt{v}{\nu}=-\nu\lh([A,v]+va), & (x,t)\in H_1\times [-1,0] \\
			v(x,t)=0, & (x,t)\in(\pt D_1\cap\Sigma)\times[-1,0].
		\end{cases}\tag{\ref{eq:eps_eng:loc_bdr_interior}'}
	\end{eq}

	Firstly, for $p=2$, \eqref{eq:eps_eng:loc_form_u_cutoff}, \eqref{eq:eps_eng:loc_bdr_interior} and the global $L^p$-estimate of Dirichlet problem on $ P_1$ implies,
	\begin{equation}\label{eq:epsreg:global_lp}
		\|\phi u\|_{W^{2,1}_p( P_1)}\leq C\norm{\phi^{1/2}|\nabla u|}_{L^{2p}( P_1)}^2+C,
	\end{equation}
	since $\|\ga\nabla u+\beta\|_{L^p( P_1)}$ is bounded.

	We claim that the same estimate holds for \eqref{eq:eps_eng:loc_form_u_cutoff} with boundary condition \eqref{eq:eps_eng:loc_bdr_boundary}. We first transform the oblique derivative boundary condition \eqref{eq:eps_eng:loc_bdr_boundary} into a homogeneous Neumann condition by multiplying the solution by a known function and then we use the reflection technique so that the required estimate follows as interior case. In fact, let us parameterize $ B_1^+$ as $(y,r)$, where $r$ is the distance from $x$ to $H_1$, and $y$ is the coordinate on $H_1$. Set $\tilde v=g_1vg_2$, where $g_1=\exp(-rA(\nu))\in C^\infty(\bar B_1^+,G)$, $g_2=\exp(-r(a(\nu)-A(\nu)))\in C^\infty(\bar B_1^+,G)$. The special choice of $g_1$ and $g_2$ is to make sure that on the boundary $H_1$,
	\[\begin{cases}
			\vppt{g_1}{\nu}=-\vppt{g_1}{r}=A(\nu)\\
			g_1=\id,
		\end{cases}\text{ and }\quad
		\begin{cases}
			\vppt{g_2}{\nu}=-\vppt{g_2}{r}=a(\nu)-A(\nu)\\
			g_2=\id,
		\end{cases}
	\]
	which implies that the boundary condition in \eqref{eq:eps_eng:loc_bdr_boundary} becomes
	\[\begin{cases}
			\tilde v(x,-1)=0,&x\in D_1\\
			\vppt{\tilde v}{\nu}=0,&\text{ on } H_1\times[-1,0]\\
			\tilde v(x,t)=0,&\text{ on }(\pt D_1\cap\Sigma)\times[-1,0].
	\end{cases}\]
	Next, we compute the equation satisfied by $\tilde v$. We shall show that it satisfies a similar equation to \eqref{eq:eps_eng:loc_form_u_cutoff} with possibly different $\alpha$ and $\beta$. In what follows, any function of the form $\alpha\nabla u+\beta$ is called a lower order term. Since $g_1$ and $g_2$ are known smooth functions, we have
	\[
		\Delta\tilde v=g_1\Delta vg_2+l.o.t.
	\]
	Thus, the equation of $\tilde v$ is
	\begin{eq}\label{eq:eps_eng:loc_form_u_var}
		\vppt{\tilde v}{t}-\Delta\tilde v
		=-g_1\phi\II(u)(\nabla u,\nabla u)g_2+\tilde\alpha\nabla u+\tilde\beta.
	\end{eq}
	Now, we define $\hat v$ be the reflection of $\tilde v$ with respect to $H_1$, i.e.,
	\[
		\hat v(x,t)=\begin{cases}
			\tilde v(x,t),&x\in B_1^+\times[-1,0]\\
			\tilde v(-x,t), &x\in (B_1\setminus B_1^+)\times[-1,0].
		\end{cases}
	\]
	Then, it's easy to verify that $\hat v$ is a weak solution of \eqref{eq:eps_eng:loc_form_u_var} on $B_1$, with Dirichlet initial-boundary condition on $\pt B_1$. Since $v=\phi u=g_1^{-1}\tilde v g_2^{-1}$ and $g_1^{-1}$, $g_2^{-1}$ are two smooth functions, the estimate \eqref{eq:epsreg:global_lp} holds as the interior case.

	Next, we claim that
	\begin{eq}
		\int_{ P_1}|\phi^{1/2}\nabla u|^{2p}\leq\eps\int_{ P_1}|D^2(\phi u)|^p+C.
	\end{eq}
	Note that
	\[
		|D^2(\phi u)|\sim\phi|D^2u|+|D^2\phi u|+|\nabla\phi\nabla u|,
	\]
	and the space-time $L^p$ norm of last two terms are bounded, we only need to show that
	\begin{eq}
		\int_{ P_1}|\phi^{1/2}\nabla u|^{2p}\leq\eps\int_{ P_1}|\phi D^2u|^p+C.
	\end{eq}
	It is a consequence of the following ``fix $t$'' version
	\begin{eq}
		\label{eq:eps_eng:fix_t}
		\int_{ D_1}\psi|\nabla u|^{2p}\leq\eps\int_{ D_1}\psi|D^2u|^p+C,
	\end{eq}
	where $\psi=[\phi(t)]^{p}$ is a cut-off in space, and we may assume it depending only on $r$.

	Since $\psi$ can be approximated by step functions, \eqref{eq:eps_eng:fix_t} can be further reduced to
	\[
		\int_{ D_R}|\nabla u|^{2p}\leq\eps\int_{ D_R}|D^2u|^p+C, \quad\forall R\in[1/2,1].
	\]
	This is a consequence of the following Galiardo-Nirenberg inequality \cite{Nirenberg1966extended}*{Thm.~1}
	\[
		\|D^j w\|_p\leq C_1\|D^mw\|_r^a\|w\|_q^{1-a}+C_2\|w\|_q,
	\]
	for any domain in $\R^n$ with cone property, where the indices satisfy
	\[
		\frac{1}{p}=\frac{j}{n}+a\bpar{\frac{1}{r}-\frac{m}{n}}+(1-a)\frac{1}{q},\quad
		0\leq j\leq m,\quad
		\frac{j}{m}\leq a\leq 1.
	\]
	Taking $w=|\nabla u|$, and
	\[
		p=2p,\quad r=p,\quad j=0,\quad n=2,\quad m=1,\quad a=1/2,\quad q=2,
	\]
	one obtains
	\begin{align*}
		\|\nabla u\|_{2p; D_R} & \leq C_1\|D^2u\|_{p; D_R}^{1/2}\|\nabla u\|_{2; D_R}^{1/2}+C_2\|\nabla u\|_{2; D_R} \\
		& \leq C\eps_0^{1/2}\|D^2u\|_{p; D_R}^{1/2}+C.
	\end{align*}
	Taking $2p$ power, we obtain the claim in case $p=2$. 

	In conclusion, \eqref{eq:epsreg:global_lp} and \eqref{eq:eps_eng:fix_t} implies that
	\begin{equation}
		\|u\|_{W^{2,1}_2(P_{1/2})}\leq C
		\label{eq:epsreg:interior_lp}
	\end{equation}

	Lastly, we will bootstrap the a prior estimate \eqref{eq:epsreg:interior_lp}. By Sobolev embedding, 
	\[
		\|\nabla u\|_{L^4( P_{1/2})}\leq C\|u\|_{W^{2,1}_2( P_{1/2})}\leq C.
	\]
	Thus one can go through the above steps with $p=4$ to bound $L^q$ norm of $|\nabla u|$, for any $q>4$.

	Repeating the above argument once again, one knows that $|\nabla u|$ is in H\"older space. The higher norm estimate follows from standard theory.
\end{proof}
\subsection{Removable of Singularity}
The removable of singularity of harmonic maps states that a smooth harmonic map defined on a punctured disk can be smoothly extended to be a smooth harmonic map on the disk as long as the energy is finite (see~\cite{SacksUhlenbeck1981existence}*{Thm.~3.6}).

Note that in the local expression \eqref{eq:loc_u}, $u_{|k} u^{-1}u_{|k}$ is just the second fundamental form of $G\embedto M_{r\times r}\cong\R^{r\times r}$ and $u\in G$, $A_k$, $\pt_kA_k$, $a_k$, $\pt_ka_k$ are all in $L^\infty(\Sigma)$, thus we can adapt the removable of singularity of \cite[Thm.~1]{LiWang2006Bubbling} to our situation to obtain the following
\begin{lem}[Removable of singularity]\label{lem:rmv_sing}
	Let $B_1$ be the unit disc in $\R^2$ and $u:B_1\setminus\set{0}\to G$ be a $W^{2,2}_{loc}$-map with finite energy satisfying \eqref{eq:loc_u}, which can be rewritten into the following equation
	\[
		\tau(u)\eqdef\Delta u-\II(u)(\nabla u,\nabla u)=\ga\nabla u+\beta,
	\]
	where $\ga\in L^\infty(B_1,\R^2)$, and $\beta\in L^p(B_1,TG)$ for some $p>2$, then $u$ extends to a map $\tilde u\in W^{2,p}(B_1,G)$.
\end{lem}
In the real application, we also need to remove boundary singularities. Since our equation is conformal invariant, we can assume that the boundary neighborhood is a half-disk. Then the boundary case can be reduced to interior one by reflection as in the proof of $\eps$-regularity (see Thm.~\ref{thm:epsreg}).
\begin{cor}[Removable of boundary singularity]\label{cor:rmv_bdy_sing}
	Let $ B_1^+$ be the upper half-disk centered at $0$, $H_1\eqdef B_1^+\cap \set{x=(x_1,x_2)\in\R^2|x_1=0}$. Recall that $A, a\in C^\infty( B_1^+,\g\otimes T^* B_1^+)$ and $\g$ be the Lie algebra of $G$. Suppose $u: B_1^+\setminus\set{0}\to G$ is a $W^{2,2}_{loc}$-map with finite energy. If $u$ satisfies
	\[\begin{cases}
			\tau(u):=\Delta u-\II(u)(\nabla u,\nabla u)=\ga\nabla u+\beta,&x\in B_1^+\\
			\vppt{u}{\nu}+[A(\nu),u]+ua(\nu)=0,&x\in H_1,
	\end{cases}\]
	where $\ga\in L^\infty( B_1^+,\R^2)$, $\beta\in L^p( B_1^+,TG)$ for some $p>2$ and  $\nu$ is the unit outer normal of $H_1$, then $u$ extends to a map $\tilde u\in W^{2,p}( B_1^+,G)$.
\end{cor}
\section{Flow up to the First Singular Time}
\label{sec:blw}
As in the case of the harmonic map flow, the flow \eqref{eq:mflw} may develop finite time singularity and the energy concentration is the cause of the singularities. In \autoref{subsec:blwcrt}, we show that as a consequence of the local energy inequality (see Prop.~\ref{prop:local_eng_ineq}) and the $\eps$-regularity (see Thm.~\ref{thm:epsreg}), away from finitely many energy concentration points, the solution converges smoothly up to the first singular time, which is (ii) of Theorem~\ref{mthm:A}. Then, we use a time-slice version of $\eps$-regularity (see Prop.~\ref{prop:timeslice_eps_reg}) to study the blow-up of the solution. In contrast to \cite{Qing1995singularities}, we are able to study the blow-up for any time sequence $t_i\to T$. Finally in \autoref{subsec:nonbb}, we show that the bubble obtained here is a harmonic sphere in $G$ and that due to the special boundary condition, there is no boundary bubble, which proves (iii) of Theorem~\ref{mthm:A}.
\begin{rmk}
	Since (i) is proved in \autoref{sec:local_existence} and (iv) follows trivially from the discussion in \autoref{sec:separation_bubble}, we shall complete the proof of Theorem \ref{mthm:A} in this section.
\end{rmk}
\subsection{The Blow-up Criterion and Finiteness of Singularities}\label{subsec:blwcrt}
We have already shown that the gauge transformation flow \eqref{eq:mflw} with smooth initial data will exist for a short time. Let $[0,T_1)$ %]
be the maximal existence interval, then we call that a point $x_0\in \Sigma$ is a \emph{singular} or \emph{energy concentrate} point at time $T_1$ and $T_1$ is called a \emph{singular time} correspondingly, if
\[
	\lim_{R\to0}\limsup_{t\uparrow T_1}\int_{D_R(x_0)}e(S(t))\geq\eps_0.
\]
Otherwise it will be called a \emph{regular} point at time $T_1$. Here $\eps_0$ is the constant in $\eps$-regularity (see Thm.~\ref{thm:epsreg}). A general principle is that if $T_1$ is maximal then there exists at least one singular point at time $T_1$, the argument is based on $\eps$-regularity and standard bootstrap technology.

Next, we will show the finiteness of singularities at $T_1$. Let $\S(S,T_1)$ be the singular set at $T_1$ which is defined by
\[
	\S(S,T_1)=\bigcap_{r>0}\set{x\in \Sigma|\limsup_{t\uparrow T_1}\eng(S(t);D_r(x))\geq\eps_0}.
\]
For any subset $\set{x_i}_{i=1}^N$ of $\S(S,T_1)$,
\[
	\limsup_{t\uparrow T_1}\eng(S(t);D_r(x_i))\geq\eps_0,\quad\forall 1\leq i\leq N,\forall r>0.
\]
We can always choose $r$ small enough, such that $\set{D_{2r}(x_i)}_{i=1}^N$ are mutually disjoint, then the local energy inequality (see Prop.~\ref{prop:local_eng_ineq}) shows that
\begin{align*}
	N\eps_0 & \leq\sum_{i=1}^N\limsup_{t\uparrow T_1}\eng(S(t);D_r(x_i)) \\
	& \leq\sum_{i=1}^N\left(\eng(S(\tau);D_{2r}(x_i))+\limsup_{t\uparrow T_1}C\frac{|t-\tau|}{r^2}\eng(S_0;\Sigma)\right) \\
	& \leq\eng\left(S(\tau);\cup_{i=1}^ND_{2r}(x_i)\right)+\frac{N}{2}\eps_0 \\
	& \leq\eng(S_0;\Sigma)+\frac{N}{2}\eps_0,
\end{align*}
for any $\tau\in\big[T_1-\frac{\eps_0r^2}{2C\eng(S_0;\Sigma)},T_1\big)$. %]
Therefore $N\leq 2\eng(S_0;\Sigma)/\eps_0$.

Finally, if $x^0\in \Sigma\setminus \S(S,T_1)$, then $S$ is smooth at $(x^0,T_1)$. In fact, since $x^0$ is a regular point at time $T_1$, there exists $r>0$ such that
\[
	\limsup_{t\uparrow T_1}\eng(S(t);D_{2r}(x^0))< \eps_0.
\]
Thus, there exists $r^0\in(0,\min\set{r,r_0}]$, where $r_0$ is the constant in $\eps$-regularity (see Thm.~\ref{thm:epsreg}), such that
\[
	\sup_{[T_1-(r^0)^2,T_1)}\eng(S(t);D_{r^0}(x^0))< \eps_0.  %]
\]
The $\eps$-regularity implies there exists a weak limit $S(\cdot, T_1)$, which is smooth away from $\S(S,T_1)$. This proves the assertion \eqref{lab:2_mthm:A} of Theorem~\ref{mthm:A}.
\subsection{Time Slice \veps-Regularity}
Applying the \emph{inverse local energy inequality} (see Prop.~\ref{prop:local_eng_ineq}), one can show that the smallness of energy at a given time-slice still small in a short time interval. Thus, one can derive a time-slice version of $\eps$-regularity from the parabolic one. We will show in a moment that, around each singular point, for \emph{any} $t_i\nearrow T_1$, one can scale the local expression of $S$ properly to satisfy the time-slice $\eps$-regularity, which will turn out to be harmonic sphere into $G$.

\begin{prop}[Time-slice $\eps$-regularity]\label{prop:timeslice_eps_reg}
Suppose $S$ is a maximal solution of \eqref{eq:mflw} with initial energy $\eng(S_0;\Sigma)< +\infty$, then there are some $\eps^0>0$ and $T'< T_1$, which depends on the particular solution, such that for any $t\in(T',T_1]$, it holds the following \emph{time-slice $\eps$-regularity}: If
\begin{eq}\label{eq:small_energy_timeslice}
	\eng(S(t);D_r(x))\leq\eps^0,
\end{eq}
for some $r< \min\set{r_0,\sqrt{T_1-T'},i_\Sigma}$, here $r_0$ (also the subsequent constant $\eps_0$) are the one as in $\eps$-regularity (see Thm.~\ref{thm:epsreg}), then there exist some $\delta_0>0$, which depends only on the initial energy and $\eps_0$, such that, for some constant $C_k$,
\[
	\sup_{\tilde D_{\delta_0r}(x)}\delta_0^kr^k|\nabla^k_{A} S|\leq C_k,\quad\forall k\geq1,
\]
where we use a tilde to distinguish the scaled set of $D_r(x)$, that is $\tilde D_{\delta_0 r}(x)\eqdef\set{\delta_0y|y\in D_r(x)}$.
\end{prop}
\begin{proof}
	Just take $\eps^0=\eps_0/3$, monotonicity of energy along the heat flow (see Prop.~\ref{prop:evolution_of_energy}) and the absolute continuity of integration imply that there is some $T''< T_1$, such that
	\[
		\int_{T''}^{T_1}\int_\Sigma|\pt_tS|^2< \eps^0/2.
	\]
	Set $T'=\frac{T_1+T''}{2}$ and $r'^2=\frac{T_1-T''}{2}$. For any fixed $t>T'$ and $r< \min\set{r_0,r',i_\Sigma}$, the inverse local energy inequality asserts that, when $s\in[t-r^2,t]\subset[T'',T_1]$, we have
	\[
		\eng(S(s);D_{r/2}(x))\leq\eng(S(t);D_{r}(x))
		+2\int_{s}^t\int_{D_{r}(x)}|\pt_tS|^2
		+4C\frac{|t-s|}{r^2}\eng(S_0;\Sigma).
	\]
	Thus if we choose $\delta_0>0$ small ( $\delta_0< 1/4$), such that $16C\delta_0^2\eng(S_0;\Sigma)< \eps^0$, then
	\[
		\sup_{s\in[t-4\delta_0^2 r^2,t)}\eng(S(s);\tilde D_{2\delta_0 r}(x))\leq\eps_0, %]
	\]
	which verifies the smallness assumption in $\eps$-regularity (see Thm.~\ref{thm:epsreg}) and the inequality follows.
\end{proof}
\subsection{Separation of Harmonic Sphere}\label{sec:separation_bubble}
As an application of $\eps$-regularity, we will show that the bubbles are harmonic maps from $S^2$ to $G$.

We restrict ourselves to the case of Euclidean metric for simplicity. Suppose that $x_0=0$ is the unique singular point in a coordinate chart $U$ at time $T_1$, locally $S(x,t)$ can be viewed as a map $u(x,t):U\times[0,T_1)\to G$. Let %]
\[
	u_i(x)=u(x,t_i),\quad x\in U, t_i\nearrow T_1.
\]
Since $u_i$ has finite energy, it is clear that $u_i\in W^{1,2}(U)\cap C^\infty(U)$ and the weak compactness implies $u_i\weakto u_\infty$ in $W^{1,2}(U)$. Set
\[
	\frac{1}{\lambda_i}=\max_{\bar U}|\nabla u(x,t_i)|=|\nabla u(x_i,t_i)|,\quad x_i\in\bar U.
\]
There are two cases, according to the type of bubbles,
\begin{enumerate}
	\item Boundary: $\frac{1}{\lambda_i}\dist(x_i,U\cap\pt \Sigma)\to\rho< +\infty$;
	\item Interior: $\frac{1}{\lambda_i}\dist(x_i,U\cap\pt \Sigma)\to+\infty$.
\end{enumerate}
Clearly, $\lambda_i\to0$, $x_i\to x_0$ in both cases. When discussing the blow-up process, we pass to subsequences without explicit indications.

We only consider the first case, since the latter (simpler) one will follow with minor modification. It should be noted that the blow-up point of the latter case maybe located at the boundary, but the bubble will be defined over $\R^2$, which is called \emph{interior} bubble, contrasted to the one defined over half plane, which is called \emph{boundary} bubble.

\begin{figure}[!htbp]
	\centering
	\includegraphics{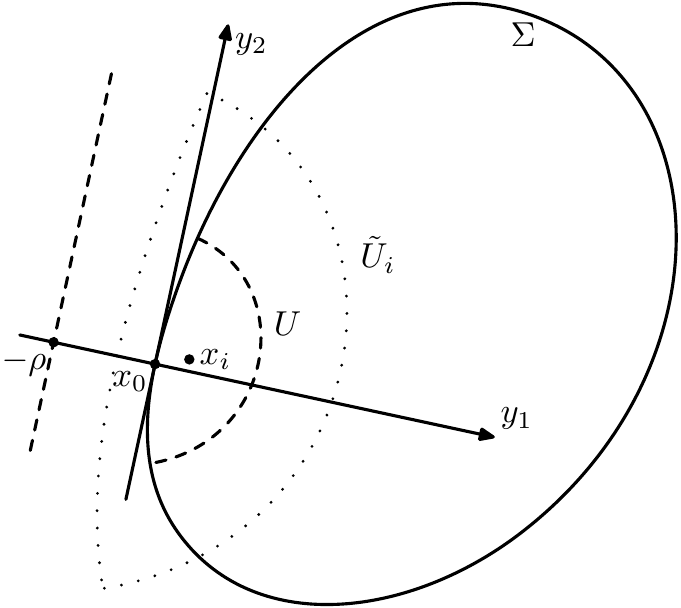}
	\caption{The Blow-up Process}
	\label{fig:blw}
\end{figure}

Firstly, one can take a good coordinate system (see~\autoref{fig:blw}), with origin at $x_0=0$, $y_1$ axis pointing to the interior of $\Sigma$ and $y_2$ axis tangent to $\pt \Sigma$ at $x_0$. Let
\begin{eq}\label{eq:scaling}
	w_i(x,t)\eqdef u(x_i+\lambda_ix,t_i+\lambda_i^2t),\quad(x,t)\in\tilde U_i\times I_i,
\end{eq}
where $\tilde U_i\eqdef\set{y\in\R^2|x_i+\lambda_i y\in U}$ and $I_i\eqdef\set{t\in\R^1|t_i+\lambda_i^2 t\in[0,T_1)}$. %]
To derive the equation of $w_i$, recall that the local expression of the flow is (see \eqref{eq:loc_form_u} in the proof of $\eps$-regularity Theorem~\ref{thm:epsreg}),
\[
	\vppt{u}{t}=\Delta u-\II(u)(\nabla u,\nabla u)+\ga\nabla u+\beta,\quad (x,t)\in U\times[0,T_1). %]
\]
Therefore, $w_i$ satisfies
\begin{eq}\label{eq:blowup:eq}
	\vppt{w_i}{t}=\Delta w_i-\II(w_i)(\nabla w_i,\nabla w_i)+\lambda_i\ga\nabla w_i+\lambda_i^2\beta.
\end{eq}
Under the scaling \eqref{eq:scaling}, the oblique boundary condition \eqref{eq:loc_bdy} becomes,
\begin{eq}\label{eq:blowup:bdr}
	0%=\nu\lh(u^{-1}(du+[A,u])+a)|_{(x_i+\lambda_ix,t_i+\lambda_i^2 t)}
	=\nu\lh(w_i^{-1}(dw_i/\lambda_i+[A,w_i])+a)|_{(x,t)},\quad
	(x,t)\in\tilde U_{\pt;i}\times I_i,
\end{eq}
where $\tilde U_{\pt;i}\eqdef\set{y\in\R^2|x_i+\lambda_i y\in\overline{U}\cap\pt\Sigma}$ is part of the boundary of $\tilde U_i$ corresponding to the boundary of $\pt\Sigma\cap \overline{U}$.

For any fixed $R>0$,
\[
	\sup_{D_R}|\nabla w_i(\cdot,0)|=\sup_{\tilde D_{\lambda_iR}(x_i)}\lambda_i|\nabla u_i|\leq 1.
\]
which implies that the time-slice small energy assumption \eqref{eq:small_energy_timeslice} in Prop.~\ref{prop:timeslice_eps_reg} is satisfied for $w_i(x,0)$, $x\in D_{R/2}$ and $r$ small. So that $w_i(\cdot,0)$ converges to $w_\infty^0(\cdot)$ smoothly locally in $\R^2_+\eqdef \set{(y_1,y_2)\in\R^2|y_1>-\rho}$. What's more, from the proof of Prop.~\ref{prop:timeslice_eps_reg}, there exists some $\delta_0>0$, such that $\sup_{[-4\delta_0^2r^2,0)%]
}\eng(w_i;\tilde D_{2\delta_0r}(x))< \eps_0$, for all $x\in D_{R/2}$ and $r$ small as above. Thus $\eps$-regularity (see Thm.~\ref{thm:epsreg}) implies that $w_i\to w_\infty$ smoothly in $P_{\delta_0R}$, and
\[
	\int_{P_{\delta_0R}}\abs{\vppt{w_\infty}{t}}^2
	=\lim_{i\to\infty}\int_{P_{\delta_0R}}\abs{\vppt{w_i}{t}}^2
	=\lim_{i\to\infty}\int_{P_{\lambda_i\delta_0R}(x_i,t_i)}
	\abs{\vppt{u}{t}}^2=0,
\]
where the last equality follows from absolute continuity of integration and
\[
	\int_0^{T_1}\int_\Sigma\abs{\vppt{S}{t}}^2< +\infty.
\]
Therefore, $w_\infty^0(\cdot)=w_\infty(\cdot,0)$ is a harmonic map from $\R^2_+$ to $G$ with finite energy. In fact $|\nabla w_i(\cdot,0)|_{D_R}\leq|\nabla\omega_i(0,0)|=1$ and $\ga\in L^\infty(D_R)$, $\gb\in L^\infty(D_R)$ clearly implies all the lower terms of $w_i$ in \eqref{eq:blowup:eq} will vanish as $i\to\infty$. By \eqref{eq:blowup:bdr}, the boundary condition of $w_\infty^0$ is the following Neumann type
\[
	\frac{\pt w_\infty}{\pt\nu}\eqdef\nu\lh dw_\infty=\lim_{i\to\infty}\nu\lh dw_i=-\lim_{i\to\infty}\lambda_i(\nu\lh [w_i,A]+w_ia)=0.
\]
Thus, one can extend $\omega_\infty^0$ to $\R^2$ by reflection. It is still a harmonic map by the regularity of $C^1$ harmonic maps. The removable of singularity of harmonic maps (see \cite{SacksUhlenbeck1981existence}*{Thm.~3.6}) implies that $\omega_\infty^0$ is a bubble, that is, a non-constant harmonic map from $S^2$ to $G$.
\subsection{Non-existence of Boundary Bubbles}\label{subsec:nonbb}
One can further rule out the boundary bubbles. In fact, $w_\infty^0$ is a non-constant harmonic map from $S^2$ to $G$. Let $\Phi$ be the corresponding holomorphic quadratic differential, it vanishes identically on $S^2$ (see, for example \cite[p.~500, Lem.~9.15]{Jost2011Riemannian}). In particular, this implies that $w_\infty^0$ is weakly conformal, i.e., $(w_\infty^0)^*h_N=\lambda(x)g_{S^2}$, $\lambda\geq0$. Since for boundary bubble, one has $\frac{\pt w_\infty^0}{\pt\nu}=0$ on the equator and so $\lambda(x)g_{S^2}(\nu,\nu) =h_N((w^0_\infty)_*(\nu),(w^0_\infty)_*(\nu)) =0$, thus $\lambda=0$ on the equator, which implies that the tangent derivative also vanishes along the equator. Finally, a classical theorem of harmonic maps asserts that a harmonic map with constant value on the boundary is a constant map (see, e.g. \cite[p.~503, Thm.~9.1.3]{Jost2011Riemannian}). In particular, we show that each singular point at least produce one bubble and this proves assertion \eqref{lab:3_mthm:A} of Theorem~\ref{mthm:A}. In conclusion, we finish the proof of Theorem~\ref{mthm:A}.

\section{Flow beyond Singular Times}\label{sec:gen_sol}
Instead of considering the flow starting from $W^{1,2}$ initial data as Struwe \cite{Struwe1985evolution}*{Thm.~4.2} and obtaining the \emph{global weak solution}, one can also start with smooth initial value and run the flow till the first singular time, then approximate the weak solution over singularities by smooth sections to obtain \emph{generalized solution}. Intuitively, when one is smoothing out the singularities, the energy of bubbles will also be excluded, thus the total energy should be decreasing along the flow also. However, the cut-off will also cost some energy. To ensure the energy of generalized solution be monotonically decreasing along the flow, one needs to control the cost delicately. Moreover, one should be careful about the boundary condition for boundary concentration points. %Finally, one can take a sequence of time-slice tends to infinity, which will give a solution of the corresponding elliptic problem.

If the solution $S(t)$ obtained in \autoref{sec:local_existence} blows up at $T_1< +\infty$, by (ii) of Theorem~\ref{mthm:A}, $S(T_1)$ is in $W^{1,2}(\Aut_GE)$. Following Struwe, one should be able to use some generalized version of \cite{SchoenUhlenbeck1982regularity}*{Lem.~3.2} to approximate $S(T_1)$ by smooth sections and study the limit of flows from these smooth sections. Here, we take a slightly different approach. More precisely, by using some oscillation estimate (see Prop.~\ref{prop:osc_estimats}), we construct the approximation explicitly (see \autoref{subsec:modify}) and restart the flow.  

The approximation depends on some parameter which we choose to be small so that the cost of energy due to the cut-off is small and the energy of the approximation is smaller (by some fixed amount) than the energy of the flow before the blow-up. The proof of Theorem~\ref{mthm:C} (see \autoref{sec:proof_mthmC}) follows by repeating the argument.

In what follows, we first show the oscillation estimate for $S(T_1)$ near the concentration points, then prove the important approximation lemma (see Lem.~\ref{mthm:B}) and finally the existence of generalized solution (see Thm.~\ref{mthm:C}) follows with little efforts.
\subsection{Oscillation Estimate around Singularity}
Although in general we do not know whether $S(T_1)$ is continuous or not near a concentration point, we do have some control over the oscillation of it in an annular region, which is important for our proof of Lemma~\ref{mthm:B}.
\begin{prop}\label{prop:osc_estimats}
	Suppose $x_0$ is an isolated singular point at time $T_1\,$ for \eqref{eq:mflw}, then for any $\eps>0$, there exists $\delta=\delta(\eps)>0$, such that
	\begin{eq}\label{eq:osc_estimates}
		\osc_{D_\delta(x_0)\setminus D_{\delta/2}(x_0)}S(T_1)\eqdef\sup_{x_1,x_2\in D_\delta(x_0)\setminus D_{\delta/2}(x_0)}|S(x_1,T_1)-S(x_2,T_1)|< \eps.
	\end{eq}
\end{prop}
\begin{proof}
	The argument is based on \cite{Qing2003remark}*{Thm.~3.4}, in which he obtained a gradient estimate of the solution around finite time singularities. Since the weak limit $S(T_1)\eqdef S(\cdot,T_1)$ has finite energy, one can take $r_0$ small, such that, for $\eps^0$ given in Proposition~\ref{prop:timeslice_eps_reg},
	\[
		\eng(S(T_1);D_{|x-x_0|}(x))\leq\eps^0, \quad \forall x\in D_{r_0}(x_0)\setminus\set{x_0}.
	\]
	Thus there exist some constants $C_2>0$ and $\delta_0>0$ (they are independent of the center $x$ and the radius $|x-x_0|$), such that,
	\[
		|\nabla^2_{A}S|(y,T_1)\leq\frac{C_2}{\delta_0^2|x-x_0|^2},\quad\forall y\in\tilde D_{\delta_0|x-x_0|}(x).
	\]

	We claim that $|\nabla_{A}S|(x,T_1)\leq\frac{o(1)}{|x-x_0|}$ as $x\to x_0$. In fact, if not, then there exist some $\bar\eps>0$ and a sequence $x_k\to x_0$ as $k\to\infty$, such that
	\[
		|x_k-x_0||\nabla_{A}S|(x_k,T_1)>\bar\eps,\quad\forall k=1,2\ldots.
	\]
	For any $y\in \tilde D_{\nu|x_k-x_0|}(x_k)$, where $\nu< \min\set{\delta_0,\frac{\delta_0^2}{2C_2}\bar\eps}$ is a fixed small number, by mean value theorem, there exists $z_k\in \tilde D_{\nu|x_k-x_0|}(x_k)$,
	\begin{align*}
		|\nabla_{A}S|(y,T_1) & =|\nabla_{A}S|(x_k,T_1)+(y-x_k)\cdot\nabla|\nabla_{A}S|(z_k,T_1)\\
		& \geq \frac{\bar\eps}{|x_k-x_0|}-|y-x_k||\nabla^2_{A}S|(z_k,T_1)\quad\text{(Kato inequality)}\\
		& \geq\frac{1}{|x_k-x_0|}\left(\bar\eps-C_2\nu/\delta_0^2\right) 
		\geq\frac{\bar\eps}{2|x_k-x_0|}.
	\end{align*}
	Integrating over $\tilde D_{\nu|x_k-x_0|}$,  we obtain
	\begin{equation}\label{eq:osc_contraction}
		\int_{\tilde D_{\nu|x_k-x_0|}(x_k)}|\nabla_{A}S(\cdot,T_1)|^2\geq C\bar\eps^2\nu^2,
	\end{equation}
	for some constant depending only on the geometry of $\Sigma$. 

	Finally, note that,
	\[
		\int_\Sigma|\nabla_{A}S(\cdot,T_1)|^2
		\leq2\int_\Sigma\dkf{|S(\cdot,T_1)a|^2+2e(S(\cdot,T_1))}< +\infty.
	\]
	The absolute continuity of integration implies that
	\[
		\int_{\tilde D_{\nu|x_k-x_0|}(x_k)}|\nabla_{A}S(\cdot,T_1)|^2\to0,\quad k\to\infty,
	\]
	which contradicts to \eqref{eq:osc_contraction}.

	Now consider $x_1,x_2\in D_\delta(x_0)\setminus D_{\delta/2}(x_0)$, let $\gc(t)$ be a path in $D_\delta(x_0)\setminus D_{\delta/2}(x_0)$ connecting $x_1$, $x_2$, and parameterize $\gc$ by arc length. It is clear that we can assume the length $L(\gc)\leq l_0\delta$, where $l_0$ is a constant depending only on the metric of $\Sigma$. Note that, one can assume further that $D_\delta(x_0)$ is contained in a trivialization neighborhood $U_\ga$, then
	\[
		(\nabla_{\dot{\gc}(t);A}S)|_{U_\ga}=\dt{S|_{U_\ga}(\gc(t))}+\dot\gc(t)\lh[A,S]|_{U_\ga},
	\]
	where $A$ is the $\g$-valued 1-form on $U_\ga$ such that $\nabla_{A}|_{U_\ga}=\rd+A$. Therefore, in $U_\ga$,
	\[
		\left|\dt{S(\gc(t))}\right|\leq|\nabla_{\dot\gc(t);A}S|+C,
	\]
	where $C$ is a constant depending on the connection $A$ and the structure group $G$. Now, take $\delta=\delta(\eps)$, such that $2Cl_0\delta< \eps$, then \eqref{eq:osc_estimates} follows from the following calculation
	\begin{align*}
		|S(x_1,T_1)-S(x_2,T_1)| & \leq\int_0^{L(\gc)}\left|\ppt{t}S(\gc(t),T_1)\right|\rd t \\
		& =\int_0^{L(\gc)}\xkf{|\nabla_{\dot\gc(t);A}S(\cdot,T_1)|+C}\rd t \\
		& \leq\int_0^{L(\gc)}\left(\frac{o(1)}{|\gc(t)-x_0|}+C\right)\rd t\\
		&\leq \xkf{\frac{2o(1)}{\delta}+C}l_0\delta< \eps.
	\end{align*}
\end{proof}
\subsection{Modification at the Singularities: the Approximation Lemma}\label{subsec:modify}
In this subsection, we will prove the approximation lemma (see Lemma~\ref{mthm:B}).

It is reasonable to divide our discussion into two cases according to the location of singularities, because at the interior singularity, in contrast to the boundary one, there is no need to worry about the boundary condition.

Firstly, let us consider the simpler interior case. Suppose that $x_0$ is an interior singular point, without loss of generality, one can assume that $D_{2\delta}(x_0)$ (which is just $B_{2\delta}(x_0)$ in interior case) is contained in a trivialization neighborhood $U$, and the local expression (under a fixed local frame and coordinate system) of $S(\cdot,T_1)|_U$ is given by a map $u:U\to G$. For any $\bar u\in G$, there is a uniform constant $\eta_1>0$ such that the exponential map at $\bar u$ is a diffeomorphism from a neighborhood of $0$ in $\g$ onto $B_{\eta_1}(\bar u)$. For this $\eta_1>0$, by Proposition~\ref{prop:osc_estimats}, one can take $\delta=\delta_1>0$ even smaller such that $u|_{D_{\delta}(x_0)\setminus D_{\delta/2}(x_0)}\in B_{\eta_1}(u^0)\subset G$, where $x^0$ is a fixed point in $D_{\delta}(x_0)\setminus D_{\delta/2}(x_0)$, $u^0=u(x^0)\in G$. 

Let $\phi$ be some cutoff function satisfying,
\[
	\phi=\begin{cases}
		0,&D_{\delta/2}(x_0)\\
		1,&D^c_{\delta}(x_0),
	\end{cases}\quad\text{and}\quad|\nabla\phi|< \frac{4}{\delta}
\]
Set 
\[
	\tilde u(x)=\exp_{u^0}(\phi(x)\exp_{u^0}^{-1}u(x))\quad\text{in}\quad D_{2\delta}(x_0),
\]
which implies that
\[
	\tilde u(x)=\begin{cases}
		u^0,&x\in D_{\delta/2}(x_0)\\
		u(x),&x\in D_{2\delta}(x_0)\setminus D_{\delta}(x_0).
	\end{cases}
\]
We claim that the gauge transformation, denoted by $\tilde S(\cdot,T_1)$, which is equal to $S(\cdot,T_1)$ outside $D_\delta(x_0)$ and is given by the local expression $\tilde u$ inside $D_\delta(x_0)$, is the required approximation in Lemma~\ref{mthm:B}. 

To see the claim is true, it suffices to estimate the cost of energy in the cutoff. For this purpose, we compute
\begin{align*}
	2\eng(\tilde S(\cdot,T_1);D_{2\delta}(x_0)) & =\int_{D_{2\delta}(x_0)}|\tilde u^{-1}(\rd\tilde u+[A,\tilde u]+a)|^2\rd x \\
	& =\zkf{\int_{D_{2\delta}(x_0)\setminus D_{\delta}(x_0)}+\int_{D_\delta(x_0)\setminus D_{\delta/2}(x_0)}+\int_{D_{\delta/2}(x_0)}}
	|\rd\tilde u+[A,\tilde u]+\tilde ua)|^2\rd x\\
	& =2\eng(S(\cdot,T_1);D_{2\delta}(x_0)\setminus D_{\delta}(x_0))
	+\int_{D_{\delta/2}(x_0)}
	|[A,u^0]+u^0a|^2\rd x\\
	& \qquad+\int_{D_\delta(x_0)\setminus D_{\delta/2}(x_0)}
	|\rd\tilde u+[A,\tilde u]+\tilde ua|^2\rd x.
\end{align*}
By our choice of $\delta$, for any $x\in D_\delta(x_0)\setminus D_{\delta/2}(x_0)$, one has $\tilde u,u\in B_\eta(u^0)$. Therefore, $|\tilde u|$ is bounded by a universal constant depending only on $G$. Also, for sufficiently small $\eta=\eta_2<\eta_1$, the differential $|d\exp_{u^0}|_{\g_\eta}|$ and $|d\exp_{u^0}^{-1}|_{B_\eta(u^0)}|$ are all bounded by a universal constant. Thus,
\begin{align*}
	|\rd\tilde u+[A,\tilde u]+\tilde ua| & \leq C+|d\tilde u|
	\leq C(1+|d\phi||\exp_{u^0}^{-1}u|+|du|)\\
	& \leq C(1+\eta/\delta+|du+[A,u]+ua|).
\end{align*}
For any given $\sigma>0$, there exist some $0< \delta^0< \delta_1$ and $0< \eta^0< \eta_2$, such that for any $\delta< \delta^0$ and $\eta< \eta^0$, one has (recall, $N$ is the number of blowup points),
\begin{gather*}
	\int_{D_\delta(x_0)\setminus D_{\delta/2}(x_0)}
	|d\tilde u+[A,\tilde u]+\tilde ua|^2\rd x
	\leq C\xkf{\delta^2+\eta^2+E(u;D_{\delta}(x_0))}< \frac{\sigma}{3N},\\
	\intertext{and}
	\int_{B_{\delta/2}(x_0)}|[A,u^0]+u^0a|^2\rd x\leq C\delta^2\leq\frac{\sigma}{3N},\quad
	\eng(S(\cdot,T_1);D_{\delta}(x_0))\leq\frac{\sigma}{3N}.
\end{gather*}
Thus,
\[
	0\leq\eng(\tilde S;D_{2\delta_0}(x_0))-\eng(S(\cdot,T_1);D_{2\delta_0}(x_0)\setminus D_{\delta_0}(x_0))\leq\frac{2\sigma}{3N},\quad\forall\delta_0< \delta^0.
\]
In particular,
\begin{eq}\label{eq:energy_diff_loc}
	\abs{\eng(\tilde S;D_{2\delta_0}(x_0))-\eng(S(\cdot,T_1);D_{2\delta_0}(x_0)}\leq\frac{\sigma}{N}.
\end{eq}
This proves the approximation lemma (see Lemma~\ref{mthm:B}) in interior case.

When the energy concentration point lies on the boundary, there are no boundary (half) bubbles as shown in \autoref{subsec:nonbb}, however there may exist interior (whole) bubble at the boundary. To take care of the boundary condition, we need a three-step modification. Firstly, let us do the modification as above, and denote the resulting gauge transformation by $S_1$.

\begin{figure}[!htbp]
	\centering
	\includegraphics{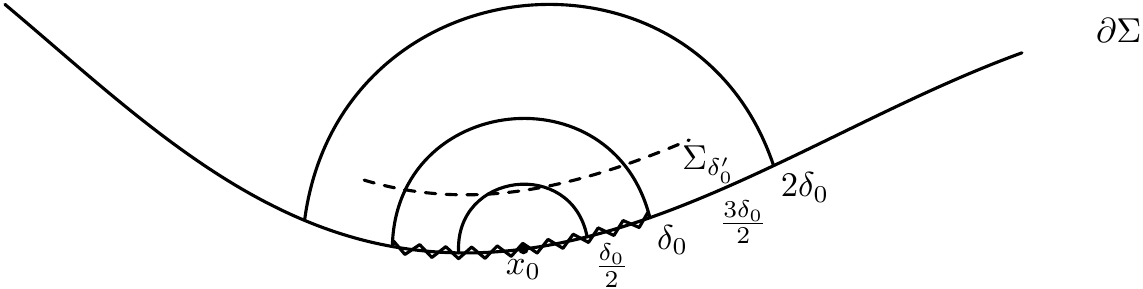}
	\caption{Cut-off at the boundary}
\end{figure}
We restrict ourselves in $D_{2\delta_0}(x_0)$, and let $u,w$ be the representation of $S(\cdot,T_1)$ and $S_1$ in $D_{2\delta_0}(x_0)$, respectively. Note that in step one, we remove the singularity by approximation, but mess up the boundary condition. In the next step, we will adjust the boundary condition by constructing a map which satisfies the boundary condition and is close enough to the map in step one. Recall our boundary condition is
\[
	\vppt{u}{\nu}+[A(\nu),u]+ua(\nu)=0.
\]
Set $V(u)\eqdef[A(\nu),u]+ua(\nu)$, it is a tangent vector of $G$ at $u$. We will use prime to denote the part of boundary on $\pt\Sigma$, e.g., $\pt'D_{2\delta_0}=\pt D_{2\delta_0}\cap\pt\Sigma$. Let $x=(y,r)$ be the coordinate near $\pt' D_{2\delta_0}$ (sufficiently close), where $r$ is the distance of $x$ to $\pt'D_{2\delta_0}$ and $y$ is the coordinate of $\pt'D_{2\delta_0}$. Then our new mapping can be expressed as
\[
	q(y,r)=\exp_{w(y,0)}(rV(w(y,0))),\quad y\in\pt'D_{2\delta_0}(x_0).
\]
It is easy to verify the following properties
\begin{itemize}
	\item $q(y,0)=w(y,0)$;
	\item The boundary condition holds for $q$. In fact, at the boundary,
		\[
			\vppt{q}{\nu}=-\vppt{q}{r}=-V(w(y,0))=-V(q(y,0)).
		\]
\end{itemize}
Finally, set
\[
	v(y,r)=\exp_{u^0}(\xi\varphi\exp_{u^0}^{-1}q+(1-\xi\varphi)\exp_{u^0}^{-1}w),
\]
where $\xi$ is a cutoff function, satisfies
\[
	\xi=\begin{cases}
		1,&x\in\Sigma_{\delta_0'}\\
		0,&x\in\Sigma\setminus\Sigma_{2\delta_0'},
	\end{cases}\quad|\nabla\xi|\leq\frac{2}{\delta_0'},\quad0\leq\xi\leq1,
\]
for some $\delta_0'$ small enough ($\delta'_0< \min\set{1,\delta_0}$), and $\Sigma_{\delta_0'}$ denotes all the points of $\Sigma$ such that its distance to $\pt\Sigma$ less than $\delta_0'$. There $\varphi=\varphi(y)$ is a cut off function on $\pt\Sigma$, such that
\[
	\varphi=\begin{cases}
		1,&y\in\pt'D_{3\delta_0/2},\\
		0,&y\in\pt\Sigma\setminus\pt'D_{2\delta_0},
	\end{cases}
	\quad|\nabla\varphi|\leq\frac{4}{\delta_0},\quad0\leq\varphi\leq1,
\]

Note that, for $y\in\pt' D_{2\delta_0}(x_0)$, $w(y,0)$ and $d\exp|_{w(y,0)}$ is bounded by $|u^0|$, $1+\eta$ and the geometry of $G$. In the following, we write $a\lesssim b$, if $a$ is less than $b$ multiplied by $1+\eta$, $\delta_0$ and a constant depending only on $A,a,u^0$, $G$ but independent of $\delta_0'$. We omit the region $\Sigma_{2\delta_0'}\cap D_{2\delta_0}(x_0)$ for $L^\infty$ norm to simplify the notation. Then
\begin{align*}
	\|\nabla q\|_{L^\infty} & \lesssim\|V(w(y,0))\|_{L^\infty} +\delta_0'\|\nabla_y[V(w(y,0))]\|_{L^\infty}\\
	&\qquad+\|\nabla_yw(y,0)\|_{L^\infty}\exp(\delta_0'\|V(w(y,0))\|_{L^\infty})\\
	& \lesssim\|w(y,0)\|_{L^\infty(\pt'D_{2\delta_0}(x_0))}+\xkf{\delta_0'+\exp(\delta_0'\eta C(A,a,u^0))}\|\nabla_yw(y,0)\|_{L^\infty(\pt'D_{2\delta_0}(x_0))}.
\end{align*}
Thus, $\exp^{-1}_{u^0}q$ can be defined for $\eta$, $\delta_0'$, $\delta_0$ small enough.

To verify the boundary condition holds for $v$, one only needs to concern on the part $\pt'D_{2\delta_0}\setminus\pt'D_{3\delta_0/2}$. On this part of boundary, one has $\pt_r(\xi\varphi)=0$, $q=w=v$, thus $V(q)=V(w)=V(v)$, and
\begin{align*}
	\vppt{v}{r}&=d\exp_{u^0}|_{\exp^{-1}_{u^0}q}\dkf{\pt_r(\xi\varphi)\zkf{\exp_{u^0}^{-1}q-\exp_{u^0}^{-1}w}
	+\xi\varphi\pt_r\exp_{u^0}^{-1}q+(1-\xi\varphi)\pt_r\exp_{u^0}^{-1}w}\\
	&=d\exp_{u^0}|_{\exp^{-1}_{u^0}q}\dkf{\xi\varphi d\exp_{u^0}^{-1}|_q\vppt{q}{r}+(1-\xi\varphi)d\exp_{u^0}^{-1}|_w\vppt{w}{r}}\\
	&=d\exp_{u^0}|_{\exp^{-1}_{u^0}q}d\exp_{u^0}^{-1}|_q\dkf{\xi\varphi V(q)+(1-\xi\varphi)V(w)}\\
	&=d[\exp_{u^0}\comp\exp_{u^0}^{-1}]|_qV(v)\\
	&=V(v),
\end{align*}
that is
\[
	\vppt{v}{\nu}=-\vppt{v}{r}=-V(v),
\]
which verifies the boundary condition.

We need to show that the energy of $v$ in $D_{2\delta_0}(x_0)$ is small enough. In fact,
\begin{gather*}
	|\nabla v|\lesssim|\nabla(\xi\varphi)\xkf{\exp_{u^0}^{-1}q-\exp_{u^0}^{-1}w}|+|\xi\varphi\nabla q|+(1-\xi\varphi)|\nabla w|,\\
	|\exp_{u^0}^{-1}q(y,r)-\exp_{u^0}^{-1}q(y,0)|\lesssim|\nabla q|r,\\
	|\exp_{u^0}^{-1}w(y,r)-\exp_{u^0}^{-1}q(y,0)|\lesssim|\nabla w|r.
\end{gather*}
Thus,
\begin{align*}
	\int_{D_{2\delta_0}(x_0)}|\nabla v|^2 & \lesssim 2\int_{D_{2\delta_0}(x_0)}|\nabla(\xi\varphi)|^2\xkf{|\nabla q|^2+|\nabla w|^2}r^2+|\xi\nabla q|^2+2\int_{D_{2\delta_0}(x_0)}|\nabla w|^2 \\
	& \lesssim 2\int_{\Sigma_{2\delta_0'}\cap D_{2\delta_0(x_0)}}2(|\nabla\xi|^2+|\nabla\varphi|^2)(|\nabla q|^2+|\nabla w|^2)\delta_0'^2+|\nabla q|^2+2\int_{D_{2\delta_0}(x_0)}|\nabla w|^2\\
	& \lesssim \zkf{1+\delta_0'^2\xkf{\frac{1}{\delta_0'^2}+\frac{1}{\delta_0^2}}}
	\dkf{
		\int_{\Sigma_{2\delta_0'}\cap D_{2\delta_0(x_0)}}|\nabla q|^2
		+\int_{D_{2\delta_0}(x_0)}|\nabla w|^2
	}\\
	&\lesssim\xkf{1+\frac{\delta_0'^2}{\delta_0^2}}
	\dkf{
		\int_{\Sigma_{2\delta_0'}\cap D_{2\delta_0(x_0)}}|\nabla q|^2
		+\int_{D_{2\delta_0}(x_0)}|\nabla w|^2
	}.
\end{align*}
We already shown, in $\Sigma_{2\delta_0'}\cap D_{2\delta_0}(x_0)$,
\begin{align*}
	|\nabla q(y,r)|&\lesssim|w(y,0)|+\xkf{\delta_0'+\exp(\delta_0'\eta C(A,a,u^0))}|\nabla_yw(y,0)|\\
	&\leq C(A,a,u^0)(1+\eta)\zkf{1+\xkf{\delta_0'+\exp\xkf{\delta_0'\eta C(A,a,u^0)}}\xkf{\frac{\eta}{\delta_0}+\frac{1}{\delta_0}}}.
\end{align*}
Since we can restrict $\eta< 1$ and $\delta_0'< 1$,
\[
	|\nabla q(x)|^2\leq C(A,a,u^0)\xkf{1+\frac{1}{\delta_0^2}},\quad\forall x\in\Sigma_{2\delta_0'}\cap D_{2\delta_0}(x_0).
\]
We first take $\delta_0$ small such that
\[
	\int_{D_{2\delta_0}(x_0)}|\nabla w|^2\leq\frac{\sigma}{9N C(A,a,u^0)}.
\]
Thus, we can take $\delta_0'$ small enough (which depend on $\delta_0$), such that
\begin{align*}
	\int_{D_{2\delta_0}(x_0)}|\nabla v|^2&\leq C(A,a,u^0)\zkf{\xkf{1+\frac{1}{\delta_0^2}}\xkf{1+\frac{\delta_0'^2}{\delta_0^2}}\delta_0\delta_0'
	+\xkf{1+\frac{\delta_0'^2}{\delta_0^2}}\int_{D_{2\delta_0}(x_0)}|\nabla w|^2
}\\
&\leq \frac{\sigma}{9N}+\frac{\sigma}{9N}
+C(A,a,u^0)\frac{\delta_0'^2}{\delta_0^2}\int_{D_{2\delta_0}(x_0)}|\nabla w|^2\leq\frac{\sigma}{3N}.
\end{align*}
Let $\tilde S$ be the gauge transformation locally represented by $v$ in $D_{2\delta_0}(x_0)$, then it is clear that \eqref{eq:energy_diff_loc} holds. Thus we finished the proof of approximation lemma.
\subsection{The Existence of Generalized Solution}\label{sec:proof_mthmC}
In what follows, we first show that for any given $\sigma< \eps_0$ (the constant in $\eps$-regularity), the energy losses in the approximation process is at least $\eps_0-\sigma$. Then we show the generalized solution described in Theorem~\ref{mthm:C} exists and its energy is monotonically decreasing.

To compute the energy losses, note that since $x_0$ is a singular point, we know that there exists a subsequence $\set{t_i}\nearrow T_1$, such that for any $\delta'< \delta_0$ ($\delta_0$ is the constant given in approximation lemma at $x_0$),
\[
	\lim_{t_i\nearrow T_1}\eng(S(\cdot,t_i);D_{\delta'}(x_0))\geq\eps_0.
\]
On the other hand, since $x_0$ is the only singularity in $D_{2\delta_0}(x_0)$, the $\eps$-regularity implies that $S(\cdot,t)$ converges strongly ($C^\infty$) to $S(\cdot,T_1)$ in $D_{2\delta_0}(x_0)\setminus D_{\delta'}(x_0)$,
\begin{align*}
	\lim_{t_i\nearrow T_1}\eng(S(\cdot,t_i);D_{2\delta_0}(x_0))
	& =\lim_{t_i\nearrow T_1}\eng(S(\cdot,t_i);D_{2\delta_0}(x_0)\setminus D_{\delta'}(x_0))+\lim_{t_i\nearrow T_1}\eng(S(\cdot,t_i);D_{\delta'}(x_0)) \\
	& \geq\eng(S(\cdot,T_1);D_{2\delta_0}(x_0)\setminus D_{\delta'}(x_0))+\eps_0.
\end{align*}
Taking $\delta'\to0$, by \eqref{eq:energy_diff_loc}, we obtain the energy cost during approximating the weak solution
\[
	\lim_{t_i\nearrow T_1}\eng(S(\cdot,t_i);D_{2\delta_0}(x_0))\geq\eng(\tilde S(\cdot,T_1);D_{2\delta_0}(x_0))+\eps_0-\sigma.
\]
Therefore,
\begin{eq}\label{eq:energy_losses}
	\lim_{t_i\nearrow T_1}\eng(S(\cdot,t_i);\Sigma)\geq\eng(\tilde S(\cdot,T_1);\Sigma)+N(T_1)(\eps_0-\sigma),
\end{eq}
where $N(T_1)$ is the number of blow-up points at time $T_1$.

Next, we turn to the existence of generalized solution. We run the flow \eqref{eq:mflw} until the first singular time, then take the approximation map $\tilde S$ (given in approximation lemma) as new initial value, and run the flow again. Proposition~\ref{prop:evolution_of_energy} shows that
\[
	\ppt{t}\eng(S(\cdot,t);\Sigma)=-\int_\Sigma|\pt_tS|^2\leq0.
\]
Let $I_k\eqdef [T_k,T_{k+1})$, %]
$k=0,1,\ldots,K$, where $T_k$ are singular times (except $T_0=0$ and $T_{K+1}=+\infty$). To simplify the notation, we write $S$ be the generalized solution, thus for singular time $T_k$, $S(T_k)$ should be understood as the $\sigma$-approximation map stated in approximation lemma (see Lemma~\ref{mthm:B}). Firstly, for $t_i\nearrow T_1$ as in \eqref{eq:energy_losses}, integrate $S$ on $[T_0,t_i]$,
\[
	\int_{T_0}^{t_i}\int_\Sigma|\pt_tS|^2=\eng(S(T_0);\Sigma)-\eng(S(t_i);\Sigma),
\]
take limits and apply \eqref{eq:energy_losses},
\[
	\eng(S(T_1);\Sigma)\leq\eng(S(T_0);\Sigma)-N(T_1)(\eps_0-\sigma).
\]
Apply similar arguments on $I_k$, $k=1,2,\ldots, K-1$,
\[
	\eng(S(T_{k+1});\Sigma)\leq\eng(S(T_k);\Sigma)-N(T_{k+1})(\eps_0-\sigma).
\]
Summing up,
\[
	\eng(S(T_K);\Sigma)\leq\eng(S_0;\Sigma)-(\eps_0-\sigma)\sum_{k=1}^{K}N(T_k).
\]
Finally, for $I_K$, since energy is monotonically decreasing on $I_K$,
\[
	\int_{T_K}^\infty\int_{\Sigma}|\pt_t S|^2\leq\eng(S(T_K);\Sigma)\leq \eng(S_0;\Sigma)-(\eps_0-\sigma)\sum_{k=1}^{K}N(T_k)< +\infty,
\]
which implies that there exists $t_i\to\infty$, such that
\[
	\|(\pt_t S)(\cdot,t_i)\|_{L^2(\Sigma)}\to0.
\]
Similar to the blow-up of a sequence of harmonic maps with tangent filed in $L^2$ (the process is almost the same as finite time blow-up discussed in \autoref{subsec:blwcrt}), we know that $S^i\eqdef S(\cdot,t_i)$ converge smoothly to some $S_\infty$ as $t_i\to\infty$ away from finite many singularities. By the removable of singularity theorem (see Lem.~\ref{lem:rmv_sing} and Cor.~\ref{cor:rmv_bdy_sing}), we conclude that $S_\infty$ solves \eqref{eq:meq} on $\Sigma$, and Theorem~\ref{mthm:C} follows.
\appendix
\section{Proof of Corollary~\ref{mthm:D}}\label{sec:app}
In \cite{Riviere2007Conservation}, the existence of Coulomb gauge with respect to a ``small'' connection on a disc plays an important role in the existence of conservation law. It is quite natural to expect ``good estimate'' under ``good gauge'', and then the existence of conservation law can be obtained from a fixed point argument of systems of elliptic partial differential equations.

One can also obtain the above existence of Coulomb gauge by heat flow method. In fact, as a result of the existence of generalized solution of \eqref{eq:mflw} when applied to a special case, we will present a flow version proof of the existence (c.f. \citelist{\cite{Schikorra2010remark}*{Thm.~2.1}\cite{FrohlichMuller2011existence}*{Prop.~2}}).

In what follows, we will take $\Sigma$ be a small disc (dimension two), not necessarily with Euclidean metric. A connection of vector bundle $E$ over $\Sigma$, where $E$ is with rank $m$ and structure group $\SO(m)$, can be written as $d+\Omega$, where $\Omega\in\so(m)\otimes T^*\Sigma$ is a 1-form with value in $\so(m)$, then for any gauge transformation $S$,
\[
	S^*(d+\Omega)=d+S^{-1}dS+S^{-1}\Omega S\eqdef d+\tilde\Omega.
\]
Thus, if one sets $A_0=d$ as a trivial connection and $A=d+\Omega$, the existence of generalized solution of \eqref{eq:mflw} implies that there exists some $S\in C^\infty(\Sigma,\SO(m))$, such that
\[
	0=\nabla^*_{d+\tilde\Omega}\tilde\Omega=\nabla_d^*\tilde\Omega +\set{\tilde\Omega,\tilde\Omega}=d^*\tilde\Omega.
\]
By Poincar\'e Lemma, there exists some $\xi\in C^\infty(\Sigma, \so(m))$, such that
\begin{eq}
	\label{eq:hodge_decomp}
	\begin{cases}
		\nabla^\perp\xi=S^{-1}dS+S^{-1}\Omega S, & x\in \Sigma\\
		\xi=0, & x\in\pt \Sigma.
	\end{cases}
\end{eq}
The boundary condition holds since, the flow implies that on $\pt \Sigma$,
\[
	0=\nu\lh(S^*(A)-A_0)=\nu\lh \tilde\Omega=\nu\lh\nabla^\perp\xi=\nu^\perp\lh(-\nabla\xi),
\]
which shows that $\xi$ is a constant along $\pt \Sigma$, and one can make $\xi$ vanishes on boundary since it is determined up to a constant.

To show the estimates hold in Corollary~\ref{mthm:D}, we need to construct $S_0$ properly. Parameterize $\Sigma_{2\delta}\eqdef\set{x\in\bar \Sigma|\dist(x,\pt \Sigma)\leq 2\delta}$as $x=(y,r)$, where $r$ is the distance to $\pt \Sigma$ and $y$ is the parameter of $\pt \Sigma$. Firstly, let us assume that $\Omega\in C^\infty(\Sigma,\so(m)\otimes T^*\Sigma)$ and set the initial value as
$$
S_0(x)=\exp\set{(1-\eta)r\Omega_\nu(x)}.
$$
where $\Omega_\nu(x)=-\Omega(r,y)(\ppt{r})$, and $\eta\in C_0^\infty(\Sigma)$, with $\eta=0$ on $\Sigma_\delta$ and $\eta=1$ on $\Sigma\setminus\Sigma_{2\delta}$, moreover, one can require $|\nabla \eta|\leq 2/\delta$. %It is a direct verification that $S_0$ is global defined since the transition functions of $\g_E$ and $\Aut_GE$ are both given by $B\mapsto g_{\ga\gb}Bg_{\ga\gb}^{-1}$.
To see the compatibility at the corner, note that, on the corner $\pt\Sigma\times\set{0}$, $S=S_0=\id$,
where $\id$ is the constant section over $\pt\Sigma$, which maps each point to the identity of $G$,
\begin{align*}
	\nu\lh(S^*(A)-A_0)&=\nu\lh(S^{-1}\rd S+S^{-1}\Omega S)=\nu\lh(dS+\Omega)\\
	&=-\ppt{r}\lh(\pt_rS_0dr)+\Omega_\nu=-\Omega_\nu+\Omega_\nu=0.
\end{align*}
The energy of $S_0$ can be controlled as
\[
	2\eng(S_0) =\int_\Sigma |S_0^{-1}dS_0+S_0^{-1}\Omega S_0|^2
	\leq \int_\Sigma(|dS_0|^2+|\Omega|^2)
	=\|dS_0\|_2^2+\|\Omega\|_2^2.
\]

We claim that $\eng(S_0)$ can be taken as small as wanted. In fact, note that for $x\in\Sigma_{2\delta}$, the derivative of $\exp$ is uniformly bounded with respect to $\delta$, thus
\[
	|dS_0|\leq C\xkf{(1-\eta)|\Omega_\nu|+r\xkf{|\nabla\eta||\Omega_\nu|+|d\Omega_\nu|}},
\]
and
\[
	|dS_0|^2\leq C|\Omega|_{1,2}^2\xkf{1+r^2|\nabla\eta|^2+r^2}.
\]
Finally,
\begin{align*}
	\int_\Sigma|dS_0|^2 & =\int_{\Sigma_{2\delta}\setminus\Sigma_{\delta}}|dS_0|^2
	\leq C\int_{\Sigma_{2\delta}\setminus\Sigma_{\delta}}
	|\Omega|_{1,2}^2(1+\delta^2\cdot 1/\delta^2+\delta^2),
\end{align*}
which tends to $0$ as $\delta\to0$ by the absolute continuity of integration.

Thus, for any $\eps>0$, one can construct $S_0$, such that $\|dS_0\|_2\leq\eps$. Since the energy is monotonically decreasing along the flow,
\[
	\|\tilde\Omega\|_2^2=2\eng(S)\leq2\eng(S_0)
	\leq\|\Omega\|_2^2+\eps^2.
\]
Moreover, since $\nabla^\perp\xi=S^{-1}dS+S^{-1}\Omega S$,
\[
	\|dS\|_2\leq\|\Omega\|_2+\|\nabla^\perp\xi\|_2 =\|\Omega\|_2+\|\tilde\Omega\|_2 \leq2\|\Omega\|_2+\eps.
\]
By Poincar\'e inequality, one concludes that
\[
	\|dS\|_2+\|\xi\|_{1,2}\leq C(m)\|\Omega\|_2.
\]
To show the required estimate for $\Omega\in L^2(\Sigma,\so(m)\otimes T^*\Sigma)$, let us approximate it by $\Omega_k\in C^\infty(\Sigma,\so(m)\otimes T^*\Sigma)$ in $L^2$ with $\|\Omega_k\|_2\leq\|\Omega\|_2$. For $A_0=\rd$, $A_k\eqdef\rd+\Omega_k$, $a_k\eqdef\Omega_k$, run the gauge transformation heat flow,
\[\begin{cases}
		S^{-1}\vppt{S}{t}=-\nabla^*_{S^*(A_k)}(S^*(A_k)-A_0), & (x,t)\in D\times(0,+\infty) \\
		S(x,0)=\exp((1-\eta)ra_{k,\nu}), & x\in D \\
		\nu\lh(S^*(A_k)-A_0)=0, & (x,t)\in\pt D\times[0,+\infty)%]
\end{cases}\]
one obtains $S_k\in W^{1,2}(D,\SO(m))\cap C^\infty$, $\xi_k\in W^{1,2}(D,\SO(m))\cap C^\infty$, which solves \eqref{eq:hodge_decomp}, with
\[
	\|dS_k\|_2+\|\xi_k\|_{1,2}\leq C(m)\|\Omega_k\|_2.
\]
Therefore, $S_k$, $\xi_k$ are $W^{1,2}$ bounded, the weak compactness implies that $S_k$, $\xi_k$ converges weakly to some $S_\infty$, $\xi_\infty$ in $W^{1,2}$, respectively. Take limits in \eqref{eq:hodge_decomp}, one finds the required gauge. The weakly lower semi-continuity implies the required estimate, and we finish the proof of Corollary~\ref{mthm:D}.

\bigskip
%\printbibliography
%\bibliography{bib/gaugeflow_boundary_bib}
\def\cprime{$'$}

\end{document}